\documentclass[leqno,11pt]{amsart}

\usepackage{geometry}

\usepackage{multirow,booktabs}

\usepackage{scalerel}

\usepackage{amsmath, amssymb, amsfonts, latexsym, mdwlist, amsthm}
\usepackage{subfig}
\usepackage{graphicx}
\usepackage{wrapfig}

\usepackage{mathtools}
\usepackage{tikz}
\usetikzlibrary{calc,trees,positioning,arrows,chains,shapes.geometric,%
	decorations.pathreplacing,decorations.pathmorphing,shapes,%
	matrix,shapes.symbols}

\tikzset{
	>=stealth',
	punktchain/.style={
		rectangle,
		rounded corners,
		draw=black, thick,
		minimum height=3em,
		text centered,
		on chain},
	line/.style={draw, thick, <-},
	element/.style={
		tape,
		top color=white,
		bottom color=blue!50!black!60!,
		minimum width=8em,
		draw=blue!40!black!90, very thick,
		text width=10em,
		minimum height=3.5em,
		text centered,
		on chain},
	every join/.style={->, thick,shorten >=1pt},
	decoration={brace},
	tuborg/.style={decorate},
	tubnode/.style={midway, right=2pt},
}

\usepackage[all]{xy} 

\usetikzlibrary{patterns}

\usepackage{geometry}

\usepackage{enumerate}

\usepackage{amsmath, amssymb, amsfonts, latexsym, mdwlist, amsthm}
\usepackage{subfig}
\usepackage{graphicx}
\usepackage{wrapfig}

\usepackage{mathtools}

\usepackage[bookmarks, colorlinks, breaklinks, pdftitle={},
pdfauthor={Soheyla Feyzbakhsh}]{hyperref}
\hypersetup{linkcolor=blue,citecolor=blue,filecolor=black,urlcolor=blue}

\usepackage{tikz}
\usetikzlibrary{calc,trees,positioning,arrows,chains,shapes.geometric,%
    decorations.pathreplacing,decorations.pathmorphing,shapes,%
    matrix,shapes.symbols}

\tikzset{
>=stealth',
  punktchain/.style={
    rectangle,
    rounded corners,
    draw=black, thick,
    minimum height=3em,
    text centered,
    on chain},
  line/.style={draw, thick, <-},
  element/.style={
    tape,
    top color=white,
    bottom color=blue!50!black!60!,
    minimum width=8em,
    draw=blue!40!black!90, very thick,
    text width=10em,
    minimum height=3.5em,
    text centered,
    on chain},
  every join/.style={->, thick,shorten >=1pt},
  decoration={brace},
  tuborg/.style={decorate},
  tubnode/.style={midway, right=2pt},
}

\usepackage{paralist}
\setdefaultenum{(a)}{(i)}{}{}
\usepackage{enumitem} 

\def\abs#1{\left\lvert#1\right\rvert}

\makeatletter
\newtheorem*{rep@theorem}{\rep@title}
\newcommand{\newreptheorem}[2]{%
\newenvironment{rep#1}[1]{%
 \def\rep@title{#2 \ref{##1}}%
 \begin{rep@theorem}}%
 {\end{rep@theorem}}}
\makeatother

\newtheorem{Thm}{Theorem}[section]
\newreptheorem{Thm}{Theorem}
\newtheorem{Prop}[Thm]{Proposition}

\newtheorem{Lem}[Thm]{Lemma}

\newreptheorem{Cor}{Corollary}

\newreptheorem{Con}{Conjecture}

\newtheorem{thm-int}{Theorem}

\theoremstyle{definition}
\newtheorem{Def-s}[Thm]{Definition}
\newtheorem{Def}[Thm]{Definition}
\newtheorem{Rem}[Thm]{Remark}

\newcommand{\ignore}[1]{}

\begin{document}

\title[Mukai's Program]{Mukai's Program (reconstructing a K3 surface \\ from a curve) via wall-crossing, II}

\author{Soheyla Feyzbakhsh}
\address{School of Mathematics, Imperial College London, Huxley Building, South Kensington Campus, SW7 2AZ, London, United Kingdom}
\email{s.feyzbakhsh@imperial.ac.uk}


\begin{abstract}
	Let $C$ be a curve on a K3 surface $X$ with Picard group $\mathbb{Z}.[C]$. Mukai's program seeks to recover $X$ from $C$ by exhibiting it as a Fourier-Mukai partner to a Brill-Noether locus of vector bundles on $C$. We use wall-crossing in the space of Bridgeland stability conditions to prove this for genus $\ge14$. 
	This paper deals with the case $g-1$ prime left over from Paper I.
\end{abstract}

\vspace{-1em}
\maketitle

\section{Introduction}\label{sec:intro}
In this paper, we consider the problem of reconstructing a K3 surface from a curve on that surface where the curve is of genus $g = p+1 \geq 14$ for a prime number $p$. The main result is the following which extends a program proposed by Mukai in \cite[Section 10]{mukai:non-abelian-brill-noether}.
\begin{Thm}\label{main.3}
	Let $(X,H)$ be a polarised K3 surface with Pic$(X) = \mathbb{Z}.H$. Let $C$ be any curve in the linear system $\abs{H}$ of genus $g \geq 14$ such that $g = p+1$ for a prime number $p$. Then $X$ is the unique K3 surface of Picard rank one and genus $g$ containing $C$, and can be reconstructed as a Fourier-Mukai partner of a certain Brill-Noether locus of vector bundles on $C$. 
\end{Thm}

Any K3 surface of Picard rank one has a canonical primitive polarisation and therefore a well-defined genus. Note that the curve $C$ in Theorem \ref{main.3} is not necessarily smooth and can be singular. Set 
\begin{equation}\label{condition m}
m \coloneqq \min\{k \in \mathbb{Z}^{>0} \colon k \nmid p+1 \}.
\end{equation} 
We consider the Brill-Noether locus $\mathcal{BN} \coloneqq M_C(m^2, 2pm, p+m^2)$ which parametrises slope semistable vector bundles on $C$ having rank $m^2$ and degree $2pm$, and possessing at least $p+m^2$ linearly independent global sections. 

Let $M_{X,H}(v)$ be the moduli space of $H$-Gieseker semistable sheaves on $X$ with Mukai vector 
\begin{equation}\label{vi}
v \coloneqq (m^2 , mH, p).
\end{equation}  
We have chosen the Mukai vector $v$ such that it is a primitive class with $v^2 =0$, hence $M_{X,H}(v)$ is a K3 surface as well. Moreover, any $H$-Gieseker semistable sheaf $E \in M_{X,H}(v)$ is $H$-slope stable\,\footnote{That is why we impose $m \nmid p+1$.}. The choice of the Brill-Noether locus $\mathcal{BN}$ is justified by the following Theorem. 
\begin{Thm}\label{theorem 1.1}
Let $(X,H)$ be a polarised K3 surface with Pic$(X) = \mathbb{Z}.H$. Let $C$ be any curve in the linear system $\abs{H}$ of genus $g \geq 14$ such that $g = p+1$ for a prime number $p$. We have an isomorphism 
	\begin{equation}\label{function}
	\psi \colon M_{X,H}(v) \rightarrow \mathcal{BN}
	\end{equation}   
	which sends a bundle $E$ on $X$ to its restriction $E|_C$.
\end{Thm}     
There exists a Brauer class $\alpha \in Br(\mathcal{BN})$ and a universal $(1 \times \alpha)$-twisted sheaf $\mathcal{E}$ on $C \times \mathcal{BN}$. Define $v' \in H^*\big(\mathcal{BN},\mathbb{Z}\big)$ to be the Mukai vector of $\mathcal{E}|_{q \times \mathcal{BN}}$ for a point $q$ on the curve $C$ (see \cite{huybrechts:equivalence-of-twisted-k3-surfaces} for definition in case $\alpha \neq 1$). The same argument as in \cite[Theorem 1.3]{feyz:mukai-program} shows that any K3 surface of Picard rank one and genus $g$ which contains $C$ is isomorphic to the moduli space $M_{\mathcal{BN},H'}^{\alpha}(v')$ of $\alpha$-twisted sheaves on $\mathcal{BN}$ with Mukai vector $v'$ which are semistable with respect to a generic polarisation $H'$ on $\mathcal{BN}$.      
The embedding of the curve $C$ into the K3 surface $M_{\mathcal{BN},H'}^{\alpha}(v')$ is given by $q \mapsto \mathcal{E}|_{q \times \mathcal{BN}}$. That is why Theorem \ref{theorem 1.1} directly implies Theorem \ref{main.3}.

\begin{Rem}
	   	In the published paper \cite{feyz:mukai-program}, we considered the Mukai vector $(4, 2H, p)$ instead of $v$ \eqref{vi}. But $H$-Gieseker stable sheaves of Mukai vector $(4, 2H, p)$ are strictly $\mu_H$-semistable, and the proof presented in \cite[Proposition 5.2.(a)]{feyz:mukai-program} is not correct, see Remark \ref{Rem.polygon} for details. Hence \cite[Theorem 1.2]{feyz:mukai-program} is not valid in case (B). To resolve the problem, we changed the class $(4, 2H, p)$ to $v=(m^2, mH, p)$ such that $m \nmid p+1$. As proven in Proposition \ref{5.2}, the condition on $m$ guarantees that any $H$-Gieseker stable sheaf of class $v$ is $\mu_H$-stable. Thus in this paper, we prove the missing cases $g=p+1$ so that eventually the main result (Theorem 1.1) in \cite{feyz:mukai-program} is proved valid.   
\end{Rem}


\subsection*{Acknowledgement}I would like to thank Arend Bayer for many useful discussions. I am grateful for comments by Daniel Huybrechts, Hsueh-Yung Lin, Richard Thomas and Yukinobu Toda. 

\section{Bridgeland stability conditions on K3 surfaces}\label{section.2}
In this section, we give a short review of the notion of Bridgeland stability conditions on the bounded derived category of coherent sheaves on a K3 surface. The main references are \cite{bridgeland:stability-condition-on-triangulated-category,bridgeland:K3-surfaces}.
\subsection{Bridgeland stability conditions} Let $(X,H)$ be a smooth polarised K3 surface with Pic$(X) = \mathbb{Z}.H$. We denote by $\mathcal{D}(X)$ the bounded derived category of coherent sheaves on the surface $X$. The
Mukai vector of an object $E \in \mathcal{D}(X)$ is an element of the lattice $\mathcal{N}(X) = \mathbb{Z} \oplus \text{NS}(X) \oplus \mathbb{Z} \cong \mathbb{Z}^3$ defined via
\begin{equation*}
v(E) = \big(\text{rk}(E),\text{c}(E)H,\text{s}(E)\big) = \text{ch}(E)\sqrt{\text{td}(X)} \in H^*(X,\mathbb{Z}),
\end{equation*}
where ch$(E)$ is the Chern character of $E$. The Mukai bilinear form 
\begin{equation*}
\left \langle v(E), v(E')\right \rangle = \text{c}(E)\text{c}(E')H^2 -\text{rk}(E)\text{s}(E') - \text{rk}(E')\text{s}(E) 
\end{equation*}
makes $\mathcal{N}(X)$ into a lattice of signature $(2,1)$. The Riemann-Roch theorem implies that this form is the negative of the Euler form, defined as
\begin{equation*}
\chi(E,E') = \sum_{i} (-1)^{i} \dim_{\mathbb{C}} \text{Hom}_X^{i}(E,E')  =  -\left \langle v(E), v(E')\right \rangle.
\end{equation*} 
Recall that for a coherent sheaf $E$ with positive rank $\text{rk}(E) >0$, the slope is defined as 
\begin{equation*}
\mu_H(E) \coloneqq \dfrac{\text{c}(E)}{\text{rk}(E)},
\end{equation*} 
and if $\text{rk}(E) = 0$, define $\mu_H(E) \coloneqq +\infty$.
\begin{Def}
	We say that an object $E \in \text{Coh}(X)$ is $\mu_H$-(semi)stable if for all proper non-trivial subsheaves $F \subset E$, we have $\mu_H(F) < (\leq) \,\mu_H(E)$.	
\end{Def}


Given a real number $b \in \mathbb{R}$, denote by $\mathcal{T}^{b} \subset \text{Coh}(X)$ the subcategory of sheaves $E$ whose quotients $E \twoheadrightarrow F$ satisfy $\mu_H(F) > b$ and by $\mathcal{F}^{b} \subset \text{Coh}(X)$ the subcategory of sheaves $E'$ whose subsheaves $F' \hookrightarrow E'$ satisfy $\mu_H(F') \leq b$. Tilting with respect to the torsion pair $(\mathcal{T}^{b},\mathcal{F}^{b})$ on Coh$(X)$ gives a bounded $t$-structure on $\mathcal{D}(X)$ with heart
\begin{equation}\label{heart}
\mathcal{A}(b) \coloneqq \{E \in \mathcal{D}(X) \colon E \cong [E^{-1}  \xrightarrow{d} E^0] , \; \text{ker } d \in \mathcal{F}^{b} \; \text{and}\; \text{cok } d \in \mathcal{T}^{b} \} \subset \mathcal{D}(X). 
\end{equation} 
For a pair $(b,w) \in \mathbb{H} = \mathbb{R} \times \mathbb{R}^{>0}$, the stability function $Z_{(b,w)} \colon \mathcal{N}(X) \rightarrow \mathbb{C}$ is defined as
\begin{equation*}
Z_{(b,w)}(r,cH,s) = \bigg\langle (r,cH,s), \bigg(1,bH, \frac{H^2}{2} (b^2 -w^2) \bigg) \bigg\rangle + i \bigg\langle (r,cH,s),\bigg(0,\frac{H}{H^2},b\bigg) \bigg\rangle.
\end{equation*}
We denote the root system by $\Delta(X) \coloneqq \{ \delta \in \mathcal{N}(X) \colon \; \langle \delta , \delta \rangle =-2 \}$.
\begin{Thm}\cite{bridgeland:K3-surfaces}  \label{Bridgeland}
	Suppose $(X,H)$ is a polarised K3 surface with Pic$(X) = \mathbb{Z}.H$. Then the pair $\sigma_{(b,w)} = \big(\mathcal{A}(b),Z_{(b,w)}\big)$ defines a Bridgeland stability condition on $\mathcal{D}(X)$ if for all $\delta \in \Delta(X)$ with rk$(\delta) >0$ and Im$[Z_{(b,w)}(\delta)] =0$ we have Re$[Z_{(b,w)}(\delta)]>0$. 
	The family of stability conditions $\sigma_{(b,w)}$ varies continuously as the pair $(b,w)$ varies in $\mathbb{H}$.
\end{Thm}
Note that the stability condition $\sigma_{(b,w)}$, up to the action of $\tilde{\text{GL}}^{+}(2,\mathbb{R})$, is the same as the stability condition defined in \cite[Section 6]{bridgeland:K3-surfaces}. We expand upon the statements in Theorem \ref{Bridgeland} by explaining the notion of $\sigma_{(b,w)}$-stability and the associated Harder-Narasimhan filtration. For a stability condition $\sigma_{(b,w)}$ and $E \in \mathcal{A}(b)$, we have $Z_{(b,w)}(v(E)) \in \mathbb{R}^{>0}\exp\big(i\pi \phi_{(b,w)}(v(E)) \big)$ where 
\begin{equation*}
\phi_{(b,w)}(v(E)) = \dfrac{1}{\pi}\tan^{-1}\bigg(-\dfrac{\text{Re}[Z_{(b,w)}(v(E))]}{ \text{Im}[Z_{(b,w)}(v(E))]  }  \bigg) + \dfrac{1}{2} \in (0,1].
\end{equation*}
We will abuse notations and write $Z(E)$ and $\phi(E)$ instead of $Z(v(E))$ and $\phi(v(E))$.
\begin{Def}
	We say that an object $E \in \mathcal{D}(X)$ is $\sigma_{(b,w)}$-(semi)stable if some shift $E[k]$ is contained in the abelian category $\mathcal{A}(b)$ and for any non-trivial subobject $E' \subset E[k]$ in $\mathcal{A}(b)$, we have $\phi_{(b,w)}(E') < (\leq) \, \phi_{(b,w)}(E[k])$.
\end{Def}
Any object $E \in \mathcal{A}(b)$ admits a Harder-Narasimhan (HN) filtration: a sequence 
\begin{equation}\label{filtration}
0=\tilde{E}_0 \subset \tilde{E}_1 \subset \tilde{E}_2 \subset ... \subset \tilde{E}_n=E
\end{equation}
of objects in $\mathcal{A}(b)$ where the factors $E_i \coloneqq \tilde{E}_i/\tilde{E}_{i-1}$ are $\sigma_{(b,w)}$-semistable and 
\begin{equation*}
\phi^{+}_{(b,w)}(E) \coloneqq \phi_{(b,w)}(E_1) > \phi_{(b,w)}(E_2) > ....> \phi_{(b,w)}(E_n) \eqqcolon \phi_{(b,w)}^{-}(E).
\end{equation*} 
In addition, any $\sigma_{(b,w)}$-semistable object $E \in \mathcal{A}(b)$ has a  Jordan-H$\ddot{\text{o}}$lder (JH) filtration into stable factors of the same phase, see \cite[Section 2]{bridgeland:K3-surfaces} for more details. 

Suppose $E_1 \hookrightarrow E_2 \twoheadrightarrow E_3 $ is a short exact sequence in $\mathcal{A}(b)$. Since $H^{i}(E_j) = 0$ for $j=1,2,3$ and $i \neq 0,-1$, taking cohomology gives a long exact sequence of coherent sheaves 
\begin{equation*}
0 \rightarrow H^{-1}(E_1) \rightarrow H^{-1}(E_2) \rightarrow H^{-1}(E_3) \rightarrow H^0(E_1) \rightarrow H^0(E_2) \rightarrow H^0(E_3) \rightarrow 0.
\end{equation*}  
For any pair of objects $E$ and $E'$ of $\mathcal{D}(X)$, Serre duality gives isomorphisms
\begin{equation*}
\text{Hom}_X^i(E,E') \cong \text{Hom}_X^{2-i}(E',E)^*.
\end{equation*} 
If the objects $E$ and $E'$ lie in the heart $\mathcal{A}(b)$, then $\text{Hom}_X^i(E,E') = 0$ if $i<0$ or $i>2$. Suppose the object $E \in \mathcal{A}(b)$ is $\sigma_{(b,w)}$-stable, then $E$ does not have any non-trivial subobject with the same phase, thus $\text{Hom}_X(E,E) = \text{Hom}_X^2(E,E)^* = \mathbb{C}$. This implies 
\begin{equation}\label{bogomolove-k3 surface}
v(E)^2 +2 = \text{Hom}^1_X(E,E) \geq 0.
\end{equation}

To simplify drawing the figures, we always consider the following projection:
\begin{equation*}
pr\colon \mathcal{N}(X) \setminus \{s = 0 \}  \rightarrow \mathbb{R}^2 \;\;,\;\; pr(r,cH,s) = \bigg(\dfrac{c}{s} , \dfrac{r}{s}\bigg).
\end{equation*}  
Take a pair $(b,w) \in \mathbb{H}$, the kernel of $Z_{(b,w)}$ is a line inside the negative cone in $\mathcal{N}(X) \otimes \mathbb{R} \cong \mathbb{R}^3$ spanned by the vector $\big(2,2bH,H^2(b^2+w^2)\big)$. Its projection is denoted by 
\begin{equation*}
k(b,w) \coloneqq pr\big(\ker Z_{(b,w)}\big) = \bigg(\dfrac{2b}{H^2(b^2+w^2)} ,  \dfrac{2}{H^2(b^2+w^2)}    \bigg).
\end{equation*}   
Thus, for any stability condition $\sigma_{(b,w)}$, we associate a point $k(b,w) \in \mathbb{R}^2$. The two dimensional family of stability conditions of form $\sigma_{(b,w)}$, is parametrised by the space 
\begin{equation*}
V(X) \coloneqq \left\{ k(b,w)\colon \;  \text{the pair }\big(\mathcal{A}(b) , Z_{(b,w)}\big) \text{ is a stability condition on } \mathcal{D}(X) \right\}  \subset \mathbb{R}^2
\end{equation*}  
with the standard topology on $\mathbb{R}^2$. 
\begin{figure}[h]
	\begin{centering}
		\definecolor{zzttqq}{rgb}{0.27,0.27,0.27}
		\definecolor{qqqqff}{rgb}{0.33,0.33,0.33}
		\definecolor{uququq}{rgb}{0.25,0.25,0.25}
		\definecolor{xdxdff}{rgb}{0.66,0.66,0.66}
		
		\begin{tikzpicture}[line cap=round,line join=round,>=triangle 45,x=1.0cm,y=1.0cm]
		
		\draw[->,color=black] (-3,0) -- (3,0);

		
		

		\fill [fill=gray!40!white] (0,0) parabola (1.6,2.56) parabola [bend at end] (-1.6,2.56) parabola [bend at end] (0,0);
		
		\draw [dashed] (0,0) parabola (1.65,2.72); 
		\draw [dashed] (0,0) parabola (-1.65,2.72);

		\draw [thick, color= white] (0.75,0.75)--(1,1);
		\draw [thick, color=white] (0.7,1.05) -- (1.5,2.25);
		\draw [thick, color=white] (0.6,1.7) -- (0.9148,2.5921);
		
		\draw [thick, color= white] (-0.75,0.75)--(-1,1);
		\draw [thick, color=white] (-0.7,1.05) -- (-1.5,2.25);
		\draw [thick, color=white] (-0.6,1.7) -- (-0.9148,2.5921);

		\draw[color=black] (0,0) -- (0,1.8);
		\draw[->,color=white] (0,1.8) -- (0,2.8);
		\draw[->,color=black] (0,2.7) -- (0,2.8);
		
		
		
		\draw [dashed, color=black] (0.7,1.05) -- (0,0);
		
		
		\draw (1.45,2.25)  node [right ] {$q_{\delta}$};	
		\draw  (0.16,1.1) node [right ] {$p_{\delta}$};
		\draw  (1.65,2.72) node [above] {$y= \frac{H^2x^2}{2}$};
		
		\draw (0,0)  node [below] {$o$};
		
		\draw (0,2.8) node [above] {$y$};
		\draw (3,0) node [right] {$x$};
		
		\begin{scriptsize}
		\fill [color=white] (0.75,0.75) circle (1.1pt);
		\fill [color=white] (0.7,1.05) circle (1.1pt);
		\fill [color=white] (0.6,1.7) circle (1.1pt);  
		
		\fill [color=white] (1,1) circle (1.1pt);
		\fill [color=white] (1.5,2.25) circle (1.1pt);
		\fill [color=white] (-1,1) circle (1.1pt);
		\fill [color=white] (-1.5,2.25) circle (1.1pt);

		
		
		\fill [color=white] (0,1.8) circle (1.1pt);
		\fill [color=white] (0.7,1.05) circle (1.1pt);
		
		\fill [color=white] (-0.75,0.75) circle (1.1pt);
		\fill [color=white] (-0.7,1.05) circle (1.1pt);
		\fill [color=white] (-0.6,1.7) circle (1.1pt);
		\fill [color=black] (0,0) circle (1.1pt);
		
		\end{scriptsize}
		
		\end{tikzpicture}
		
		\caption{The grey area is the 2-dimensional subspace of stability conditions $V(X)$.}
		
		\label{hole}
		
	\end{centering}
\end{figure}
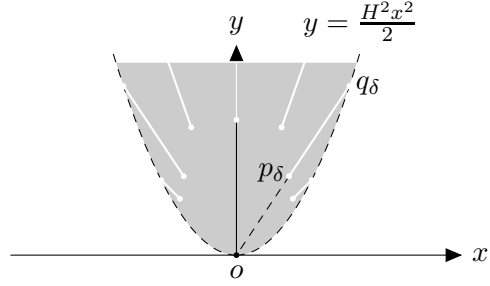	
\begin{Lem}\cite[Lemma 2.4]{feyz:mukai-program} 
	We have
	\begin{equation*}
	V(X) = \left\{ (x,y) \in \mathbb{R}^2 \colon \;\; y>\frac{H^2x^2}{2} \right\} \; \setminus \bigcup_{\delta \in \Delta(X)} I_{\delta}
	\end{equation*}
	where $I_{\delta}$ is the closed line segment that connects $p_{\delta} \eqqcolon pr(\delta)$ to $q_{\delta}$ which is 
	the intersection point of the parabola $y=\frac{H^2}{2}x^2$ with the line through the origin and $p_{\delta}$, see Figure \ref{hole}.
	
\end{Lem} 

\begin{Rem}\label{plane of having the same phase}
	The point $k(b,w)$ is on the line $x=by$. As $w$ gets larger, the point $k(b,w)$ gets closer to the origin. 
	Take two non-parallel vectors $u,v \in \mathcal{N}(X) \otimes \mathbb{R} \cong \mathbb{R}^3$, then $Z_{(b,w)}(v)$ and $Z_{(b,w)}(u)$ are aligned if and only if the kernel of $Z_{(b,w)}$ in $\mathcal{N}(X) \otimes \mathbb{R}$ lies on the plane spanned by $v$ and $u$, i.e. the points corresponding to $\mathbb{R}.u$, $\mathbb{R}.v$ and $\ker Z_{(b,w)}$ in the projective space $\mathbb{P}_{\mathbb{R}}^2$ are collinear. This, in particular, implies that if three objects $E_1,E_2$ and $E_3$ in $\mathcal{D}(X)$ have the same phase with respect to a stability condition $\sigma_{(b,w)}$, there must be a linear dependence relation among the vectors $v(E_1), v(E_2)$ and $v(E_3)$ in $\mathcal{N}(X) \otimes \mathbb{R}$.
\end{Rem} 
The 2-dimensional family of stability conditions parametrised by the space $V(X)$ admits a chamber decomposition for any object $E \in \mathcal{D}(X)$.
\begin{Prop}\label{line wall} \cite[Proposition 2.6]{feyz:mukai-program}
	Given an object $E \in \mathcal{D}(X)$, there exists a locally finite set of \emph{walls} (line segments) in $V(X)$ with the following properties:
	\begin{itemize*}
		\item[(a)] The $\sigma_{(b,w)}$-(semi)stability or instability of $E$ is independent of the choice of the stability condition $\sigma_{(b,w)}$ in any chamber (which is a connected component of the complement of the union of walls).
		\item[(b)] When $\sigma_{(b_0,w_0)}$ is on a wall $\mathcal{W}_E$, i.e. the point $k(b_0,w_0) \in \mathcal{W}_E$, then $E$ is strictly $\sigma_{(b_0,w_0)}$-semistable.
		\item[(c)] If $E$ is semistable in one of the adjacent chambers to a wall, then it is unstable in the other adjacent chamber.
		\item[(d)] Any wall $\mathcal{W}_E$ is a connected component of $L\cap V(X)$, where $L$ is a line that passes through the point $pr(v(E))$ if $\text{s}(E) \neq 0$, or that has a slope of $\text{rk}(E)/\text{c}(E)$ if $\text{s}(E) = 0$.
	\end{itemize*}
\end{Prop}
	Take an object $E \in \mathcal{D}(X)$, let $L_1$ be a connected component of $L \cap V(X)$ where $L$ is a line as described in Proposition \ref{line wall}, part $(d)$. Suppose $E$ is $\sigma_{(b_0,w_0)}$-(semi)stable for a stability condition $\sigma_{(b_0,w_0)}$ on $L_1$. Then the structure of walls shows that $E$ is (semi)stable with respect to all stability conditions on $L_1$. Moreover, if $E$ is in the heart $\mathcal{A}(b_0)$, by a straightforward computation, one can show that when we deform the stability condition $\sigma_{(b_0,w_0)}$ along the line segment $L_1$, the phase of $E$ is fixed so it remains in the heart.

\section{An upper bound for the number of global sections}\label{section.3}
In \cite[section 3]{feyz:mukai-program} we introduced a new upper bound for the number of global sections of objects in $\mathcal{D}(X)$. In this section, we provide a slight improvement of this bound, which is crucial in the later section.  

We always assume $X$ is a smooth K3 surface with Pic$(X) = \mathbb{Z}.H$. Given an object $E \in \mathcal{D}(X)$, we denote its Mukai vector by $v(E) = \big(\text{rk}(E),\text{c}(E)H,\text{s}(E) \big)$. A small modification of the proof of \cite[Lemma 3.2]{feyz:mukai-program} gives the following. 
\begin{Lem}\label{bound for h}\textbf{(Brill-Noether wall)}
	Let $\sigma_{(b_0,w_0)}$ be a stability condition with $b_0<0$ and $k(b_0,w_0)$ sufficiently close to the point $pr\big(v(\mathcal{O}_X)\big) = (0,1)=o'$. Let $E \in \mathcal{D}(X)$ be a $\sigma_{(b_0,w_0)}$-semistable object with the same phase as the structure sheaf $\mathcal{O}_X$. Define $k \coloneqq \gcd (\text{rk}(E) -s(E) , c(E))$. Then 
	\begin{equation}\label{bound for h.in}
	h^0(X,E) \leq \frac{\chi(E)}{2} + \frac{\sqrt{(rk(E) -s(E) )^2 + 2H^2c(E)^2 + 4k^2  }}{2}
	\end{equation}
	where $h^0(X,E) = \text{dim}_{\;\mathbb{C}}\,\text{Hom}_{\,X}(\mathcal{O}_X,E)$ and $\chi(E) = \text{rk}(E) +\text{s}(E)$ is the Euler characteristic of $E$. 
\end{Lem}
\begin{proof}
	If the object $E$ satisfies c$(E) = 0$, then the projection $pr(v(E))$ lies on the $y$-axis. Remark \ref{plane of having the same phase} implies that $\mathcal{O}_X$ cannot have the same phase as $E$ with respect to $\sigma_{(b_0,w_0)}$ with $b_0 <0$ unless $pr(v(E)) = pr(v(\mathcal{O}_X))$, i.e. $v(E) = kv(\mathcal{O}_X)$. Thus the uniqueness of spherical sheaf with Mukai vector $(1,0,1)$ (see e.g. \cite[Corollary 3.5]{mukai:modili-of-bundles-on-k3-surfaces}) implies that $E$ is the direct sum of $k$-copies of $\mathcal{O}_X$, hence the inequality \eqref{bound for h.in} holds. Thus we may assume c$(E) \neq 0$. 
	
	Let $L_E$ be the line through $o'$ which passes the point $pr(v(E))$ if s$(E) \neq 0$, or it has slope $\text{rk}(E)/\text{c}(E)$ if s$(E) = 0$. By the assumption, $k(b_0,w_0)$ is on the line $L_E$. Consider the evaluation map $\text{ev} \colon \text{Hom}_X(\mathcal{O}_X,E) \otimes \mathcal{O}_X \rightarrow E$.
	As shown in the proof of \cite[Lemma 3.2]{feyz:mukai-program}, the structure sheaf $\mathcal{O}_X$ is $\sigma_{(b_0,w_0)}$-stable, so it is a simple object in the abelian category of semistable objects with the same phase as $\mathcal{O}_X$. Therefore, the morphism ev is injective and the cokernel cok$(\text{ev})$ is $\sigma_{(b_0,w_0)}$-semistable. Let $\{E_i\}_{i=1}^{i=n}$ be the Jordan-H$\ddot{\text{o}}$lder factors of cok$(\text{ev})$ with respect to the stability condition $\sigma_{(b_0,w_0)}$. By Remark \ref{plane of having the same phase}, the Mukai vector of any factor can be written as $v(E_i)= m_iv(\mathcal{O}_X)+t_iv(E)$ for some $m_i, t_i \in \mathbb{R}$. It is proved in \cite[Lemma 3.2]{feyz:mukai-program} that $t_i \geq 0$ and   	
	\begin{equation}\label{sum}
	\sum_{i=1}^{n} t_i =1 \ .
	\end{equation}
	If $t_i = 0$, then since $v(E_i)^2 \geq -2$, we have $m_i=1$ so the uniqueness of spherical sheaf again implies $E_i \cong \mathcal{O}_X$. We have 
	\begin{equation*}
	\text{rk}(E_i) = m_i + t_i\,\text{rk}(E) \in \mathbb{Z} \;\;\; \text{and} \;\;\; s(E_i) = m_i + t_i\, s(E), 
	\end{equation*}
	thus $\text{rk}(E_i) -s(E_i)= t_i \big( \text{rk}(E)-s(E) \big) \in \mathbb{Z}$. Moreover, $c(E_i) = t_ic(E) \in \mathbb{Z}$, hence
	\begin{equation*}
	t_i .\gcd (\text{rk}(E) -s(E) , c(E)) \in \mathbb{Z}.  
	\end{equation*}
	Combing this with \eqref{sum} proves that the maximum number of factors with $t_i \neq 0$ is equal to $k \coloneqq \gcd (\text{rk}(E) -s(E) , c(E))$. 
	
	By reordering of the factors, we can assume $E_i \cong \mathcal{O}_X$ for $1 \leq i \leq i_0$ and the other factors satisfy $t_i \neq 0$. Therefore, 
	\begin{equation*}
	v(E)-\big(h^0(X,E)+i_0\big)v(\mathcal{O}_X) = \sum_{i=i_0+1}^{n} w_i
	\end{equation*} 
	where $0 \leq n-i_0 \leq k$. Since $\langle w_i, w_j \rangle \geq -2$ for $1 \leq  i,j \leq n$, 
	\begin{equation*}
	\big(v(E)-(h^0(X,E)+i_0)v(\mathcal{O}_X) \big)^2 = \bigg(\sum_{i=i_0+1}^{n} w_i \bigg)^2 \geq -2k^2.
	\end{equation*}
	Therefore
	\begin{equation}\label{better bound in lemm}
	h^0(X,E)\leq h^0(X,E) +i_0 \leq \dfrac{\chi(E)}{2} + \dfrac{\sqrt{ \big(\text{rk}(E)-\text{s}(E)\big)^2 +2H^2\text{c}(E)^2+4k^2 }}{2}\ .
	\end{equation}
\end{proof}

\begin{Def}
	Given a stability condition $\sigma_{(b,w)}$ and an object $E \in \mathcal{A}(b)$, the Harder-Narasimhan polygon of $E$ 
	is the convex hull of the points $Z_{(b,w)}(E')$ for all subobjects $E '\subset E$ of $E$. 
\end{Def}
If the Harder-Narasimhan filtration of $E$ is the sequence
\begin{equation*}
0 = \tilde{E}_0 \subset \tilde{E}_1 \subset .... \subset \tilde{E}_{n-1} \subset \tilde{E}_n =E,
\end{equation*}
then the points $\left\{ p_i = Z_{(b,w)}(\tilde{E}_i) \right\}_{i=0}^{i=n}$ are the extremal points of the Harder-Narasimhan polygon of $E$ on the left side of the line segment $\overline{oZ_{(b,w)}(E)}$, see Figure \ref{polygon figure.1}. 


\begin{figure} [h]
	\begin{centering}
		\definecolor{zzttqq}{rgb}{0.27,0.27,0.27}
		\definecolor{qqqqff}{rgb}{0.33,0.33,0.33}
		\definecolor{uququq}{rgb}{0.25,0.25,0.25}
		\definecolor{xdxdff}{rgb}{0.66,0.66,0.66}
		
		\begin{tikzpicture}[line cap=round,line join=round,>=triangle 45,x=1.0cm,y=1.0cm]
		
		\draw[->,color=black] (-2.3,0) -- (2.3,0);
		
		\filldraw[fill=gray!40!white, draw=white] (0,0) --(-1.3,.4)--(-1.7,1)--(-1.3,1.8)--(.5,2.5)--(2,2.5)--(2,0);

		\draw [ color=black] (0,0)--(-1.3,.4);
		\draw [color=black] (-1.3,.4)--(-1.7,1);
		\draw [color=black] (-1.7,1)--(-1.3,1.8);
		\draw [color=black] (-1.3,1.8)--(.5,2.5);
		\draw [color=black] (2,2.5) -- (.5,2.5);
		\draw [color=black, dashed] (0,0) -- (.5,2.5);
		
		\draw[->,color=black] (0,0) -- (0,3.3);

		\draw (2.3,0) node [right] {Re$[Z_{(b,w)}(-)]$};
		\draw (0,3.3) node [above] {Im$[Z_{(b,w)}(-)]$};
		\draw (0,0) node [below] {$o$};
		\draw (-1.3,.4) node [left] {$p_1$};
		\draw (-1.7,1) node [left] {$p_2$};
		\draw (-1.3,1.8) node [above] {$p_3$};
		\draw (0,2.8) node [right] {$p_4 = Z_{(b,w)}(E)$};

		\begin{scriptsize}
		
		\fill [color=black] (-1.3,.4) circle (1.1pt);
		\fill [color=black] (-1.7,1) circle (1.1pt);
		\fill [color=black] (-1.3,1.8) circle (1.1pt);
		\fill [color=black] (.5,2.5) circle (1.1pt);
		\fill [color=uququq] (0,0) circle (1.1pt);
		
		\end{scriptsize}
		
		\end{tikzpicture}
		
		\caption{The HN polygon is in the grey area.} 
		
		\label{polygon figure.1}
		
	\end{centering}
	
\end{figure}
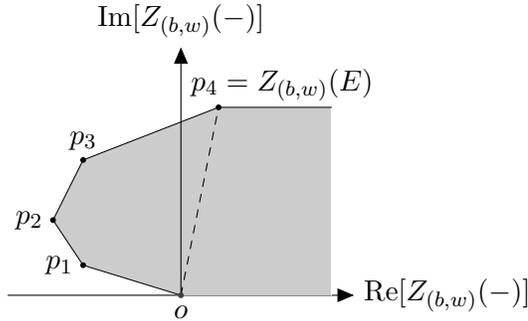

We define the following non-standard norm on $\mathbb{C}$: 
\begin{equation}\label{definition of norm}
\lVert x+iy \rVert = \sqrt{x^2 + (2H^2+4)y^2}\ .
\end{equation} 
For two points $p$ and $q$ on the complex plane, the length of the line segment $\overline{pq}$ induced by the above norm is denoted by $\lVert \overline{pq} \rVert$. The function $\overline{Z} \colon K(X) \rightarrow \mathbb{C}$ is defined as $$\overline{Z}(E) = Z_{\left(0,\sqrt{2/H^2}\right)}(E) =  \text{rk}(E)-\text{s}(E) \,+\, i\,\text{c}(E).$$ The next proposition shows that we can bound the number of global sections of an object in $\mathcal{A}(0)$ via the length of the Harder-Narasimhan polygon at some limit point.  
\begin{Prop} \cite[Proposition 3.4]{feyz:mukai-program} \label{polygon}
	Consider an object $E \in \mathcal{A}(0)$ which has no subobject $F \subset E$ in $\mathcal{A}(0)$ with ch$_1(F) = 0$. 
	\begin{itemize*}
		\item[(a)] There exists $w^* > \sqrt{2/H^2}$ such that the Harder-Narasimhan filtration of $E$ is a fixed sequence
		\begin{equation*}\label{HN}
		0 = \tilde{E}_{0} \subset \tilde{E}_{1} \subset .... \subset \tilde{E}_{n-1} \subset  \tilde{E}_n=E,
		\end{equation*}
		for all stability conditions $\sigma_{(0,w)}$ where $\sqrt{2/H^2}< w < w^*$.
		\item[(b)] Let $p_i \coloneqq  \overline{Z}(\tilde{E_i})$ for $0 \leq i \leq n$, and $P_E$ be a polygon whose sides are $\overline{p_np_0}$ and $\overline{p_ip_{i+1}}$ for $i=0, \dots, n-1$. Then $P_E$ is a convex polygon and 
		\begin{equation*}
		h^0(X,E) \leq \dfrac{\chi(E)}{2}  + \dfrac{1}{2} \sum_{i=0}^{n-1} \lVert \overline{p_ip_{i+1}} \rVert.
		\end{equation*} 
	\end{itemize*} 	
\end{Prop}

\begin{Rem}\label{Rem.polygon}
	We may slightly improve the upper bound in Proposition \ref{polygon}, part (b) by applying the following technique. Pick two vertices $p_j = a_j+i \ b_j$ and $p_k = a_k+i\ b_k$ of the polygon $P_{E}$ where $j<k$. Then 
	$$
	\sum_{i=j+1}^{k}h^0(E_i) \leq \left\lfloor \dfrac{\sum_{i=j+1}^{k}\text{rk}(E_i)+\text{s}(E_i)}{2} +\frac{\sum_{i=j+1}^{k} \lVert p_ip_{i-1} \rVert}{2} \right\rfloor.   
	$$ 
	\begin{itemize*}
		\item [i.] If $a_k - a_j$ is odd, then $\sum_{i=j+1}^{k} \text{rk}(E_i) - s(E_i)$ and so $\sum_{i=j+1}^{k} \text{rk}(E_i) + s(E_i)$ are odd, thus 
		$$
		\sum_{i=j+1}^{k}h^0(E_i) \leq \frac{\sum_{i=j+1}^{k} \chi(E_i) -1}{2 } + \left\lfloor \dfrac{1}{2} +\frac{\sum_{i=j+1}^{k} \lVert p_ip_{i-1} \rVert}{2} \right\rfloor.   
		$$
		\item [ii.] If $a_k - a_j$ is even, then $\sum_{i=j+1}^{k} \text{rk}(E_i) + s(E_i)$ is even, hence 
		$$
	   \sum_{i=j+1}^{k}h^0(E_i) \leq \frac{\sum_{i=j+1}^{k} \chi(E_i)}{2 } + \left\lfloor \frac{\sum_{i=j+1}^{k} \lVert p_ip_{i-1} \rVert}{2}\right\rfloor.   
		$$
	\end{itemize*}
\end{Rem}

\section{The Brill-Noether loci}\label{section.5}

In this section, we prove Theorem \ref{theorem 1.1}. We first show that the morphism $\psi \colon M_{X,H}(v) \rightarrow \mathcal{BN}$ described in \eqref{function} is well-defined. Then we consider a slope semistable rank $m^2$-vector bundle $F$ on the curve $C$ of degree $2pm$ and describe the location of the wall that bounds the Gieseker chamber for the push-forward of $F$. Finally, we show that if the number of global sections of $F$ is high enough, then it must be the restriction of a vector bundle in $M_{X,H}(v)$.   

We assume $X$ is a K3 surface with Pic$(X) = \mathbb{Z}.H$ and $H^2= 2p$ for a prime number $p \geq 13$. Let $C$ be any curve in linear system $|H|$ and $i \colon C \hookrightarrow X$ be the embedding of $C$ into the surface $X$. Recall that a vector bundle $F$ on $C$ is slope (semi)stable if and only if for all non-trivial subsheaves $F' \subset F$, we have 
\begin{equation*}
\frac{\chi(\mathcal{O}_C, F')}{\chi(\mathcal{O}_p, F')} < (\leq)\  \frac{\chi(\mathcal{O}_C, F)}{\chi(\mathcal{O}_p, F)}
\end{equation*} 
where $\mathcal{O}_p$ is the skyscraper sheaf at a generic point $p$ on the curve $C$. Therefore $F$ is slope (semi)stable if and only if $i_*F$ is $H$-Gieseker (semi)stable, see \cite[Section 2]{feyz:mukai-program} for more details. Define
\begin{equation}\label{condition m.2}
m \coloneqq \min\{k \in \mathbb{Z}^{>0} \colon k \nmid p+1 \}.
\end{equation} 
We will study slope semistable vector bundles $F$ on $C$ of rank $m^2$ and degree $2pm$. The push-forward of $F$ to the K3 surface $X$ has Mukai vector  
$$
v(i_*F) = \left(0,m^2H,2pm -m^2p\right).
$$
We also consider the moduli space $M_{X,H}(v)$ of $H$-Gieseker semistable sheaves on $X$ with Mukai vector 
$$v \coloneqq \left(m^2,mH,p\right).$$
Note that 
\begin{equation}\label{bound on m}
m < \frac{p-1}{2} -1.
\end{equation}
Otherwise by the definition \eqref{condition m.2}, $\frac{p-1}{2} -2 | p+1$ i.e. $\gcd\left(\frac{p-1}{2} -2, p+1\right) = \frac{p-1}{2} -2$. Since 
$$
\gcd\left(\frac{p-1}{2} -2, \, p+1\right) \leq 2 \gcd\left(\frac{p-1}{2} -2, \frac{p+1}{2}\right) \leq 6,
$$
we get $\frac{p-1}{2} -2 \leq 6$, i.e. $p \leq 17$. But $(p,m) = (13, 3)$ and $(17, 4)$ satisfy \eqref{bound on m}. This, in partiucalr, implies $v$ is a primitive class, so $M_{X,H}(v)$ is a smooth projective K3 surface by \cite[Proposition 10.2.5 and Corollary 10.3.5]{huybretchts:lectures-on-k3-surfaces}.  
\begin{Rem}
	In the published paper \cite{feyz:mukai-program}, we considered the Mukai vector $(4, 2H, p)$ instead of $v$. However $H$-Gieseker stable sheaves of Mukai vector $(4, 2H, p)$ are strictly $\mu_H$-semistable, and the proof presented in \cite[Proposition 5.2.(a)]{feyz:mukai-program} is not correct. The problem is due to the fact that there could be non-zero morphisms between $\mu_H$-stable sheaves of the same slope. As a result, we cannot apply \cite[Lemma 2.15]{feyz:mukai-program} to $H$-Gieseker stable sheaves of class $(4, 2H, p)$ and so \cite[Theorem 1.2]{feyz:mukai-program} is not valid in case (B). That is why we changed this class to $v=(m^2, mH, p)$ such that $m \nmid p+1$. As we prove in Proposition \ref{5.2}, the condition on $m$ guarantees that any $H$-Gieseker stable sheaf of class $v$ is $\mu_H$-stable.   
\end{Rem}
    
Let $u \coloneqq v(i_*F) - v = \left(-m^2 , (m^2-m)H , -p(m-1)^2\right)$, then 
\begin{equation*}
p_v \coloneqq pr(v)= \left(\frac{m}{p} , \frac{m^2}{p}\right),\;\; p_u \coloneqq pr(u) = \left(-\frac{m}{p(m-1)} , \frac{m^2}{p(m-1)^2}\right) \;\;\text{and} \;\; q \coloneqq \left(-\frac{1}{m} , \frac{p}{m} \right).
\end{equation*}
 \begin{Lem}\label{lem. no psheircal in grey region}
	There is no projection of roots in the grey area in Figure \ref{fig.2} and on the open line segments $(\overline{qp_u})$, $(\overline{op_u})$, $(\overline{op_v})$ and $(\overline{qo'})$.   

	\begin{figure} [h]
		\begin{centering}
			\definecolor{zzttqq}{rgb}{0.27,0.27,0.27}
			\definecolor{qqqqff}{rgb}{0.33,0.33,0.33}
			\definecolor{uququq}{rgb}{0.25,0.25,0.25}
			\definecolor{xdxdff}{rgb}{0.66,0.66,0.66}
			
			\begin{tikzpicture}[line cap=round,line join=round,>=triangle 45,x=1.0cm,y=1.0cm]
			
			\filldraw[fill=gray!40!white, draw=white] (0,0) --(-1.5,3/5)--(-1.89,1.26)--(-5,6.6)--
			(3,12/5)--(0,0);
			
			\draw[->,color=black] (-5.7,0) -- (5.7,0);
			\draw[color=black] (0,0) -- (0,4);
			\draw[color=black] (0,4) -- (0,7.3);
			\draw[->,color=black] (0,7.3) -- (0,7.5);
			
			\draw [] (0,0) parabola (-5.2,7.14); 
			\draw [] (0,0) parabola (5.2,7.14);
			
			\draw [color=black] (0,0) --(-1.5,3/5);

			\draw(5.2,7.14) node [above] {$y = px^2$};

			\draw (0,7.5) node [above] {$y$};
			\draw  (5.7,0) node [right] {$x$};
			\draw (0,0) node [below] {$o$};
			\draw  (1.4,3.8) node [above] {$o' = pr\big(v(\mathcal{O}_X)\big)$};
			\draw  (3,12/5) node [right] {$p_v = \left(\frac{m}{p} , \frac{m^2}{p} \right)$};
			\draw  (-3.2,.9) node [below] {$\left(\frac{m}{p(1-m)}, \frac{m^2}{p(1-m)^2} \right)=p_u$};
			\draw  (-4.95,6.6) node [left] {$\left(-\frac{1}{m} , \frac{p}{m^2} \right)=q$};

			
			\draw [color=black] (-1.5,3/5)--(-5,6.6);
			\draw [color=black]  (3,12/5)--(-5,6.6);
			\draw [color=black]  (3,12/5)--(0, 0);
			
			\draw [color=black, dashed] (0,0) --(-5,6.6);
			
			\begin{scriptsize}

			\fill [color=black] (0,0) circle (1.1pt);
			
			\fill [color=black] (-1.5,3/5) circle (1.1pt);

			\fill [color=black] (-5,6.6) circle (1.1pt);

			\fill [color=black] (3,12/5) circle (1.1pt);
			
			\fill [color=white] (0,4) circle (1.7pt);

			\end{scriptsize}
			
			\end{tikzpicture}
			
			\caption{No projection of roots in the grey area}
			
			\label{fig.2}
			
		\end{centering}
		
	\end{figure}
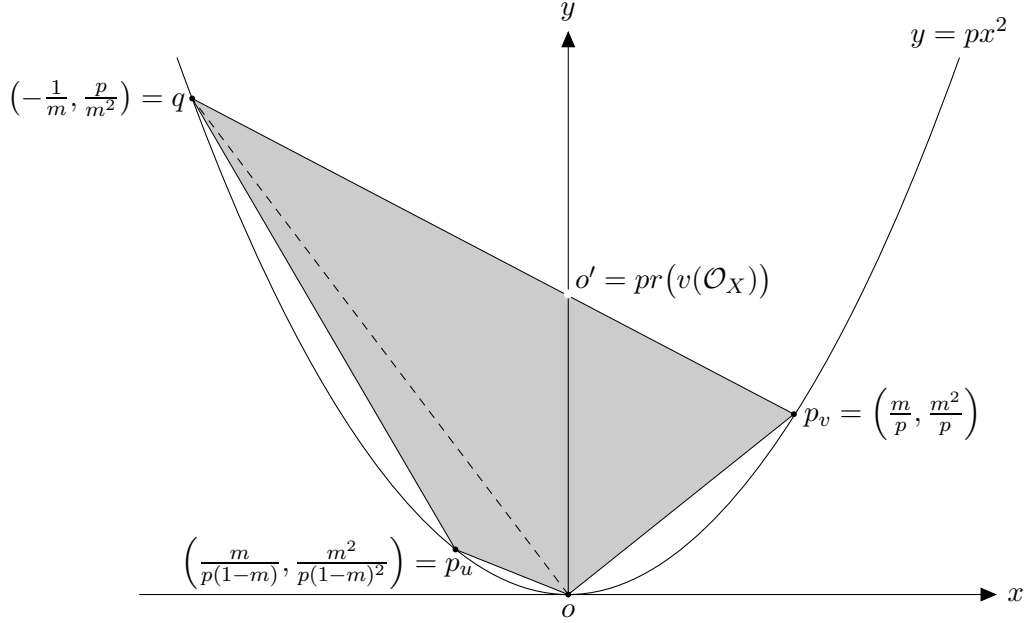

\end{Lem}

\begin{proof}
	Assume for a contradiction that projection of a root $\delta = (\tilde{r}, \tilde{c} H, \tilde{s})$ lies in the claimed region. 
	Then one of the following cases holds:  
	\begin{enumerate}
		\item $pr(\delta) = \left( \frac{\tilde{c}}{\tilde{s}} , \frac{\tilde{r}}{\tilde{s}}\right)$ lies on the right hand side of the grey region, i.e. 
		\begin{equation}\label{conditions}
		0< \frac{\tilde{c}}{\tilde{s}} < \frac{m}{p} \;\;\; \text{and} \;\;\; 
		\frac{\tilde{r}}{\tilde{c}} > m. 
		\end{equation}
		Since $\tilde{c}^2p -\tilde{r} \tilde{s} =-1$, we have $\tilde{r}\tilde{s} >0$, thus the first inequality in \eqref{conditions} gives 
		\begin{equation}\label{in.first.1}
		\tilde{c}\tilde{r} < \frac{m}{p} \tilde{s}\tilde{r} = \frac{m}{p} \left( \tilde{c}^2 p+1  \right) = \tilde{c}^2m + \frac{m}{p}. 
		\end{equation}
		The second inequality in \eqref{conditions} implies $\tilde{c}^2m < \tilde{c}\tilde{r}$. Combining it with \eqref{in.first.1} gives
		\begin{equation*}
		\tilde{c}\tilde{r} - \frac{m}{p} < \tilde{c}^2m < \tilde{c}\tilde{r}
		\end{equation*}
		which is not possible because $\frac{m}{p} <1$. 
		\item $pr(\delta)$ lies on the line segment $\overline{op_v}$, i.e. $\frac{\tilde{r}}{\tilde{c}} = m$, then $\tilde{c}^2p - m\tilde{c}\tilde{s} =-1$. Thus $\tilde{c} = \pm 1$, and $m | p+1$ which is not possible by \eqref{condition m}. 
		
		\item $pr(\delta)$ lies on the line segment $(\overline{o'q})$ or in the grey region above the line segment $\overline{oq}$. But \cite[Lemma 2.8]{feyz:mukai-program} shows this cannot happen.
		
	    \item $pr(\delta)$ is on line segment $(\overline{oq})$, then 
	    \begin{equation*}
	    \frac{\tilde{r}}{\tilde{c}} = -\frac{p}{m}. 
	    \end{equation*}   
		Since $\gcd(p, m) = 1$, we get $\tilde{r} = kp$ and $\tilde{c} = -km$ for an integer $k \in \mathbb{Z}$, so $-1=\tilde{c}^2p - \tilde{r}\tilde{s} = pk^2m^2 - pk\tilde{s}$ which is not possible.  
		
		\item $pr(\delta)$ is on the line segment $(\overline{qp_u})$ or inside the grey region between two line segments $\overline{op_u}$ and $\overline{oq}$, then
		\begin{equation}\label{in.4}
		\frac{\frac{p}{m^2} - \frac{\tilde{r}}{\tilde{s}} }{ -\frac{1}{m} - \frac{\tilde{c}}{\tilde{s}}} \geq -\frac{p}{m} - \frac{m}{m-1}\;\;\; \text{and} \;\;\;-\frac{p}{m} < \frac{\tilde{r}}{\tilde{c}} < -\frac{m}{(m-1)}.  
		\end{equation}
		The first condition in \eqref{in.4} gives 
		\begin{equation}\label{in.5}
		-\frac{p}{m} - \frac{m}{m-1} \leq \frac{\frac{p}{m^2} + \frac{p\tilde{c}}{m\tilde{s}} -\frac{p\tilde{c}}{m\tilde{s}} - \frac{\tilde{r}}{\tilde{s}} }{ -\frac{1}{m} - \frac{\tilde{c}}{\tilde{s}}} = - \frac{p}{m} + \frac{p\tilde{c} + \tilde{r}m}{\tilde{s} + \tilde{c}m}. 
		\end{equation}
		We know  $- \frac{1}{m} < \frac{\tilde{c}}{\tilde{s}} < 0$. Multiplying it by positive number $m\tilde{r}\tilde{s}$ gives $\tilde{r}(m\tilde{c} + \tilde{s}) >0$. Thus multiplying \eqref{in.5} by $\tilde{r}(m\tilde{c} + \tilde{s})$ gives   
		\begin{align*}
		(m-1)\tilde{r}(p\tilde{c} + \tilde{r}m) &\geq -m\tilde{r}(\tilde{s} + \tilde{c}m) \\
		&= -m(p\tilde{c}^2+1+ \tilde{r}\tilde{c}m)\\
		&= -m\tilde{c}(p\tilde{c} + \tilde{r}m) -m. 
		\end{align*}
		This shows $-m \leq (p \tilde{c} + \tilde{r}m) \big( \tilde{r}(m-1) + m\tilde{c}\big)$. Dividing by $-\tilde{c}^2m(m-1)$ gives
		\begin{equation} \label{in.6}
		M(\tilde{r} , \tilde{c}) \coloneqq \left(\frac{p}{m} + \frac{\tilde{r}}{\tilde{c}}\right) \left( -\frac{\tilde{r}}{\tilde{c}}- \frac{m}{m-1} \right) \leq \frac{1}{\tilde{c}^2(m-1)}. 
		\end{equation}
		By the second inequality of \eqref{in.4}, $M(\tilde{r} , \tilde{c})$ is positive. It is minimum if either $\frac{p}{m} + \frac{\tilde{r}}{\tilde{c}}$ or $- \frac{\tilde{r}}{\tilde{c}}- \frac{m}{m-1}$ gets minimum. The minimum value of $\frac{p}{m} + \frac{\tilde{r}}{\tilde{c}}$ is $\frac{1}{|\tilde{c}|m}$ which is achieved if $\frac{\tilde{r}}{\tilde{c}} = -\frac{p}{m}+ \frac{1}{|\tilde{c}|m}$, then $M(\tilde{r}, \tilde{c})$ is equal to $\frac{1}{|\tilde{c}|m} \left(- \frac{m}{m-1} + \frac{p}{m} - \frac{1}{|\tilde{c}|m}\right)$. Similarly, the minimum value of $-\frac{\tilde{r}}{\tilde{c}}- \frac{m}{m-1}$ is $\frac{1}{|\tilde{c}|(m-1)}$ which is obtained if $\frac{\tilde{r}}{\tilde{c}} = - \frac{m}{m-1} - \frac{1}{|\tilde{c}|(m-1)}$, then $M(\tilde{r}, \tilde{c})$ is equal to $\frac{1}{|\tilde{c}|(m-1)} \left(\frac{p}{m} - \frac{m}{m-1} - \frac{1}{|\tilde{c}|(m-1)} \right)$. Therefore
		\begin{equation*}
		M(\tilde{r} , \tilde{c}) \geq \min\left\{  \frac{1}{|\tilde{c}|m} \left(\frac{p}{m}- \frac{m}{m-1} - \frac{1}{|\tilde{c}|m}\right), \frac{1}{|\tilde{c}|(m-1)} \left(\frac{p}{m} - \frac{m}{m-1} - \frac{1}{|\tilde{c}|(m-1)} \right) \right\}. 
		\end{equation*}
		Since $m \geq 3$, the inequality \eqref{bound on m} gives $m+1 + \frac{1}{m-1} < \frac{p-1}{2}$. This implies 
		$$
		\frac{p}{m} - \frac{m}{m-1} > \frac{m}{m-1} + \frac{1}{m} \geq \frac{m}{|\tilde{c}|(m-1)} + \frac{1}{|\tilde{c}|m}\ .  
		$$
		Comparing the first and the last sentences gives 
		\begin{equation}\label{first part}
		\frac{1}{|\tilde{c}|m} \left( \frac{p}{m} - \frac{m}{m-1} - \frac{1}{|\tilde{c}|m}\right) > \frac{1}{\tilde{c}^2(m-1)}\ .
		\end{equation}  
		Similarly, we have 
		\begin{equation*}
		\frac{p}{m} - \frac{m}{m-1} > 1 + \frac{1}{m-1} \geq \frac{1}{|\tilde{c}|} + \frac{1}{|\tilde{c}|(m-1)} 
		\end{equation*}
		which gives
		\begin{equation}\label{second part}
		\frac{1}{|\tilde{c}|(m-1)} \left(\frac{p}{m} - \frac{m}{m-1} - \frac{1}{|\tilde{c}|(m-1)} \right) > \frac{1}{\tilde{c}^2(m-1)}\ . 
		\end{equation}
		Therefore \eqref{first part} and \eqref{second part} implies $M(\tilde{r} , \tilde{c}) > \frac{1}{\tilde{c}^2(m-1)}$ which contradicts \eqref{in.6}.

		\item $pr(\delta)$ lies on the line segment $\overline{op_u}$, then
		\begin{equation*}
		\frac{\tilde{r}}{\tilde{c}} = \frac{m}{1-m}\ .
		\end{equation*}
		Since $\gcd(m, m-1) =1$, there is an integer $k \in \mathbb{Z}$ such that $\tilde{r} = km$ and $\tilde{c} = k(1-m)$. Thus $-1 = \tilde{c}^2p - \tilde{r}\tilde{s} = k^2(m-1)^2p -km\tilde{s}$. Hence $k= \pm 1$ and 
		\begin{equation*}
		m^2p-2mp \pm m\tilde{s}   +  p =-1. 
		\end{equation*}
		This gives $m|p+1$, a contradiction. 
		
	\end{enumerate}
	
\end{proof}

\begin{Prop}\label{5.2}
	Let $E$ be an $H$-Gieseker stable sheaf on $X$ of Mukai vector $v = (m^2, mH, p)$.
	\begin{itemize}
		\item[(a)] The sheaf $E$ is $\mu_H$-stable and locally free. 
		\item[(b)] Hom$_X\big(E,E(-H)[1]\big) = 0$.
		\item[(c)] The restricted bundle $E|_C$ is a slope stable vector bundle on $C$ with $h^0(C,E|_C) =  p+m^2$. In particular, the morphism $\psi$ described in \eqref{function} is well-defined. 
		\item[(d)] The cone $K_E$ of the evaluation map 
		\begin{equation}\label{evalutaion map}
		\mathcal{O}_X^{\oplus h^0(E)} \xrightarrow{\text{ev}} E \rightarrow K_E\ ,
		\end{equation}
		is of the form $K_E = E'[1]$ where $E'$ is a $\mu_H$-stable locally free sheaf on $X$. Moreover $\text{Hom}_X\big(E', E(-H)[1]\big) = 0$.
	\end{itemize}
\end{Prop} 
\begin{proof}
	Since $E$ is $H$-Gieseker stable, it is $\mu_H$-semistable. Suppose $E$ is strictly $\mu_H$-semistable, then there is a quotient $E \twoheadrightarrow E'$ with Mukai vector $(r, cH, s)$ which has the same slope as $E$, i.e. $\frac{c}{r} = \frac{1}{m}$. We may assume $E'$ is $\mu_H$-stable, thus $v(E')^2 \geq -2$. By Lemma \ref{lem. no psheircal in grey region}, there is no projections of roots on the line segment $(\overline{op_u})$, so there is no roots of the same slope as $E$. This implies $v(E') ^2  = 2pc^2 - 2 rs \geq 0$ which gives
	\begin{equation*}
	\frac{s}{r} \leq \frac{pc^2}{r^2} = \frac{p}{m^2}. 
	\end{equation*} 	
	But $E$ is $H$-Gieseker stable, thus $\frac{s}{r} > \frac{p}{m^2}$, a contradiction. Therefore $E$ is $\mu_H$-stable and its double dual $E^{\vee\vee}$ is also $\mu_H$-stable \cite[page $156$]{huybretchts:lectures-on-k3-surfaces}. Hence
	\begin{equation*}
	-2 \leq v(E^{\vee\vee})^2 = v(E)^2 -2 \text{rk}(E)l(E^{\vee\vee}/E)= -8 l(E^{\vee\vee}/E)
	\end{equation*}
	which shows $E$ must be a locally free sheaf and proves part $(a)$.

	
	By \cite[Proposition 14.2]{bridgeland:K3-surfaces}, the coherent sheaf $E$ is $\sigma_{(0,w)}$-stable for $w \gg 0$, i.e. when $k(0,w)$ is close to the origin $o$. The next step is to show that there is no wall for $E$ in the interior of the grey region in Figure \ref{fig.walls}. Since Hom$(E, E) = \text{Hom}(E, E[2]) = \mathbb{C}$ and $v(E)^2 = 0$, we find $\dim_{\mathbb{C}} \text{Ext}^1(E, E) =2$. Therefore $E$ is a semirigid object in the sense of \cite[Definition 2.3]{bayer:derived-auto}. Suppose for a contradiction that there is a wall $\mathcal{W}_E$ as shown in Figure \ref{fig.walls}.

	\begin{figure} [h]
		\begin{centering}
			\definecolor{zzttqq}{rgb}{0.27,0.27,0.27}
			\definecolor{qqqqff}{rgb}{0.33,0.33,0.33}
			\definecolor{uququq}{rgb}{0.25,0.25,0.25}
			\definecolor{xdxdff}{rgb}{0.66,0.66,0.66}
			
			\begin{tikzpicture}[line cap=round,line join=round,>=triangle 45,x=1.0cm,y=1.0cm]
			
			\filldraw[fill=gray!40!white, draw=white] (0,0) --(-1.5,3/5)--(-1.89,1.26)--(-5,6.6)--
			(3,12/5)--(0,0);
			
			\draw[->,color=black] (-5.7,0) -- (5.7,0);
			\draw[color=black] (0,0) -- (0,4);
			\draw[color=black] (0,4) -- (0,7.3);
			\draw[->,color=black] (0,7.3) -- (0,7.5);
			
			\draw [] (0,0) parabola (-5.2,7.14); 
			\draw [] (0,0) parabola (5.2,7.14);
			
			\draw [color=black] (0,0) --(-1.5,3/5);

			\draw(5.2,7.14) node [above] {$y = px^2$};

			\draw (0,7.5) node [above] {$y$};
			\draw  (5.7,0) node [right] {$x$};
			\draw (0,0) node [below] {$o$};
			\draw  (1.1,3.88) node [above] {$o'$\fontsize{8}{2}{$= pr\big(v(\mathcal{O}_X)\big)$}};
			\draw  (3,12/5) node [right] {$p_v = \left(\frac{m}{p} , \frac{m^2}{p} \right)$};
			\draw  (-3.2,.9) node [below] {$\left(\frac{m}{p(1-m)}, \frac{m^2}{p(1-m)^2} \right)=p_u$};
			\draw  (-4.95,6.6) node [left] {$\left(-\frac{1}{m} , \frac{p}{m^2} \right)=q$};
			
			\draw  (-1.5,2.7) node [above] {\fontsize{8}{2}{$\mathcal{W}_{E}$}};
			\draw   (-3.8,2.94) node [below] {\fontsize{8}{2}{$pr(v(E_1))$}};

			\draw  (1.55, 1.9) node [above] {\fontsize{8}{2}{$\mathcal{W}_{E(-H)[1]}$}};
			
			\draw   (2.8,3.76) node [above] {\fontsize{8}{2}{$pr(v(F_1))$}};
			
			\draw   (.2, 1.35) node [below] {$\tilde{o}$};
			
			\draw [color=black] (-1.5,3/5)--(-5,6.6);
			\draw [color=black]  (3,12/5)--(-5,6.6);
			\draw [color=black, dashed]  (4.33,4.7)--(1.9, 3);
			\draw [color=black]  (1.9, 3)--(-1.5,3/5);

			\draw [color=black]  (3,12/5)--(0, 0);
			
			\draw [color=black, dashed]  (3,12/5)--(-1.5,3/5);
			
			\draw[color=black, dashed](0, 0)--(-5, 6.6);
			\draw [color=black, dashed]  (-2.77, 2.88)--(-4, 3);
			\draw [color=black]  (3,12/5)--(-2.77, 2.88);
			
			
			\begin{scriptsize}

			\fill [color=black] (0,0) circle (1.1pt);
			
			\fill [color=black] (-1.5,3/5) circle (1.1pt);

			\fill [color=black] (-5,6.6) circle (1.1pt);
			
			\fill [color=black] (0,1.2) circle (1.1pt);
			
			\fill [color=black] (3,3.76) circle (1.1pt);

			\fill [color=black] (3,12/5) circle (1.1pt);
			
			\fill [color=white] (0,4) circle (1.7pt);

			\fill [color=black] (-3.2,2.94) circle (1.1pt);
			\end{scriptsize}
			
			\end{tikzpicture}
			
			\caption{Hypothetical walls for $E$ and $E(-H)$ in the grey area}
			
			\label{fig.walls}
			
		\end{centering}
		
	\end{figure}

 By Proposition \ref{line wall}, $\mathcal{W}_E$ is a line segment ending at $p_v$. \cite[Lemma 2.5 (b)]{bayer:derived-auto} implies that at least one of the stable factors $E_1$ along the wall, is a rigid object, i.e. $\dim_{\mathbb{C}} \text{Ext}^1(E_1, E_1) = 0$. Stability of $E_1$ gives $v(E_1)^2 = -2 \dim_{\mathbb{C}}\text{Hom}(E_1, E_1) = -2$, hence $v(E_1)$ is a root class. Let $v(E_1) = (r_1, c_1H, s_1)$. By Proposition \ref{line wall}, $pr(v(E_1)) = \left(\frac{c_1}{s_1} , \frac{r_1}{s_1}\right)$ lies on the line along $\mathcal{W}_E$. Moreover, it is above the curve $y= px^2$ and not in the grey area, by Lemma \ref{lem. no psheircal in grey region}. Therefore $\mu_H(E_1) = \frac{c_1}{r_1} <0$. For any stability condition $\sigma_{(b,w)}$ on $\mathcal{W}_E$, we have 
	\begin{equation}\label{limit}
	0< \text{Im}[Z_{(b,w)}(E_1)] = c_1-r_1b < \text{Im}[Z_{(b,w)}(E)] = m-m^2b. 
	\end{equation}
	If we move the stability condition $\sigma_{(b,w)}$ along $\mathcal{W}_E$ such that $k(b,w)$ moves toward $p_v$, i.e. $b \rightarrow \frac{1}{m}$, then the right hand side in \eqref{limit} goes to zero. Therefore, $c_1/r_1 - \frac{1}{m}$ goes to zero which is not possible since $c_1/r_1 < 0$. Therefore there is no wall for $E$ in the interior of the grey area in Figure \ref{fig.walls}.

	Since $E$ is $\mu_H$-stable, the twist $E(-H)$ is also $\mu_H$-stable, so \cite[Lemma 2.15]{feyz:mukai-program} implies that $E(-H)[1]$ is $\sigma_{(b,w)}$-stable where $b = \frac{1-m}{m}$ and $w$ is arbitrary. We know $E(-H)$ is also a semirigid object, so we may apply the same argument as above to show that there is no wall for $E(-H)[1]$ in the interior of the grey region in Figure \ref{fig.walls}, or on the line segment $\overline{qp_u}$. Suppose for a contradiction that there is such a wall $\mathcal{W}_{E(-H)[1]}$. It is a line segment ending at $p_u$ and crossing the line segment $oq$. The Mukai vector of at least one of the stable factors $F_1$ is a root class, by \cite[Lemma 2.5 (b)]{bayer:derived-auto}. Since $pr(v(F_1))$ lies on the line along $\mathcal{W}_{E(-H)[1]}$  and cannot be in the grey region, it must lie on the right hand side of the line segment $\overline{oq}$, i.e. $\frac{c(F_1)}{r(F_1)} > -\frac{m}{p}$. When we move a stability condition $\sigma_{(b,w)}$ along the wall $\mathcal{W}_{E(-H)[1]}$ towards $p_u$, i.e. $b \rightarrow \frac{1-m}{m}$, then Im$[Z_{(b,w)}(F_1)] = c(F_1)-br(F_1) \rightarrow c(F_1) + \frac{m-1}{m} r(F_1)$ must go to zero, a contradiction.

	Therefore, $E$ and $E(-H)[1]$ are stable with respect to the stability condition at the point $\tilde{o} = \left(0, \frac{m^2}{p(m-1)} \right)$ where they have the same phase. Hence, the same argument as in \cite[Proposition 4.1]{feyz:mukai-program} implies that $E|_C$ is slope-stable and Hom$_X(E,E(-H)[1]) = 0$ which proves $(b)$ and the first claim of part $(c)$.
	
	Since $\gcd(p-m^2 , m) =1$, applying Lemma \ref{bound for h} for $E$ gives
	\begin{equation*}
	h^0(E) \leq \left\lfloor \frac{p+m^2}{2} + \frac{\sqrt{4pm^2 + (p-m^2)^2 +4}}{2} \right\rfloor=  \left\lfloor \frac{p+m^2}{2} + \frac{\sqrt{ (p+m^2)^2 +4}}{2} \right\rfloor = p+m^2.
	\end{equation*}
	We have $\chi(E) = p+m^2 = h^0(X,E) - h^1(X,E)\leq  h^0(X,E)$, so $h^0(X,E) = p+m^2$. On the other hand, since there is no wall for $E(-H)[1]$ in the grey region, the structure sheaf $\mathcal{O}_X$ is not making a wall for it which means Hom$_X(\mathcal{O}_X , E(-H)[1]) = 0$. Therefore $h^0(C,E|_C) = m^2+p$. This completes the proof of $(c)$.

	To prove part $(d)$, we consider a stability condition $\sigma_1$ on the open line segment $(\overline{qo'})$. Since the structure sheaf $\mathcal{O}_X$ and $E$ are $\sigma_1$-semistable of the same phase, the cokernel $K_E$ of the evaluation map is also $\sigma_1$-semistable. Suppose $K_E$ is strictly $\sigma_1$-semistable, then the same argument as in \cite[Proposition 4.1]{feyz:mukai-program} implies that it has a subobject $K_1$ with Mukai vector $v(K_1) = t_1v(E) + s_1v(\mathcal{O}_X)$ where $0 \leq t_1 \leq 1$. Therefore $\text{c}(K_1) = t_1\text{c}(E) = t_1m \in \mathbb{Z}$ and $s(k_1) -r(K_1) = t_1(p-m^2) \in \mathbb{Z}$. Since $\gcd(m, p-m^2) =1$, we have $t_1= 0,1$ which means $\mathcal{O}_X$ is either a subobject or a quotient of $K_E$. But $
	\text{Hom}_X(K_E,\mathcal{O}_X) = \text{Hom}_X(\mathcal{O}_X,K_E) = 0 $, a contradiction. Therefore, $K_E$ is $\sigma_1$-stable. 
	
	The Mukai vector of $K_E$ is $(p, -mH, m^2)$, thus $v(K_E)^2 = 0$ and $pr(v(K_E)) = q$. Stability of $K_E$ with respect to $\sigma_1$ implies that Hom$_X(E, E) = \text{Hom}_X(E, E) \cong \mathbb{C}$. Thus $K_E$ is a semirigid object, and we can apply the same argument as for $E$ to show that there is no wall for $K_E$ in the interior of the grey region. First assume there is a wall $\mathcal{W}_{K_E}$ between line segments $\overline{qo'}$ and $\overline{qo}$ in Figure \ref{fig.walls}. It must end at $q$ and cross the line segment $\overline{oo'}$. We know the Mukai vector of one the stable factors $K_1$ along the wall is a root class and $pr(v(K_1))$ lies on the line along $\mathcal{W}_{K_E}$, so $c(K_1)/r(K_1) >0$. When we move a stability condition $\sigma$ along the wall towards $q$, i.e. $b \rightarrow -\frac{m}{p}$, then $c(K_1) -br(K_1)$ must go to zero which is not possible. Therefore $K_E$ is $\sigma_{(0,w)}$-stable where $w \gg 0$, and \cite[Lemma 6.18]{macri:intro-bridgeland-stability} shows that $K_E \cong E'[1]$ for a $\mu_H$-stable locally free sheaf $E'$. Note that the rank of $E'$ and its degree $\frac{\text{ch}(-).H}{H^2}$ are co-prime, so $E'$ cannot be strictly $\mu_H$-semistable. Hence $E'[1]$ is stable with respect to the stability conditions on the line segment $\overline{qo}$, by Lemma \cite[Lemma 2.15]{feyz:mukai-program}. 
	
	Finally, we show that there is no wall for $E'$ between line segments $\overline{qo}$ and $\overline{qp_u}$ in the grey region, or on the line segment $\overline{qp_u}$ itself. If there exists such a wall, it crosses the line segment $\overline{op_u}$. Moreover, the Mukai vector of at least one of the stable factors $E'_1$ is a root class, so $pr(v(E'))$ must be below $\overline{op_u}$, i.e. $\frac{c(E')}{r(E')} \leq -\frac{(m-1)}{m}$. Then by moving along the wall towards the point $q$, we reach again a contradiction. Therefore $E'$ is stable with respect to the stability conditions on $(\overline{qp_u})$ and it has the same phase as $E(-H)[1]$. Hence there is no non-trivial homomorphism between them. This completes the proof of $(d)$.

\end{proof}

\subsection{The first wall}
Recall that $M_C(m^2,2pm)$ denotes the moduli space of slope semi-stable vector bundles $F$ of rank $m^2$ and degree $2pm$ on the curve $C$. 
\cite[Lemma 2.13]{feyz:mukai-program} implies that $i_*F$ is semistable with respect to the stability conditions $\sigma_{b,w}$ in the Gieseker chamber, which means $w \gg 0$ and $b$ is arbitrary. The next proposition describes the wall that bounds the Gieseker chamber for $i_*F$. Note that a wall for $i_*F$ is part of a line which goes through $pr(v(i_*F)) = \left(0, \frac{m}{p(2-m)} \right)$, by \cite[Proposition 2.6 (d)]{feyz:mukai-program}. 
\begin{Prop}\label{5.3}
	Given a vector bundle $F \in M_C(m^2, 2pm)$, the wall that bounds the Gieseker chamber for $i_*F$ is not below the line segment $\overline{p_up_v}$ and it coincides with the line segment $\overline{p_up_v}$ if and only if $F$ is the restriction of a vector bundle $E \in M_{X,H}(v)$ to the curve $C$. 
\end{Prop}
\begin{proof}
	Assume the wall $\mathcal{W}_{i_*F}$ that bounds the Gieseker chamber for $i_*F$ is below or on the line segment $\overline{p_up_v}$. Take the line $\ell$ aligned with $\mathcal{W}_{i_*F}$. Its intersection point with $\overline{op_u}$ and $\overline{op_v}$ are denoted by $q_1$ and $q_2$, respectively, see Figure \ref{wall.43}. 
	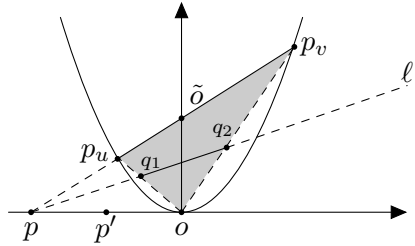
\begin{figure}[h]
		\begin{centering}
			\definecolor{zzttqq}{rgb}{0.27,0.27,0.27}
			\definecolor{qqqqff}{rgb}{0.33,0.33,0.33}
			\definecolor{uququq}{rgb}{0.25,0.25,0.25}
			\definecolor{xdxdff}{rgb}{0.66,0.66,0.66}
			
			\begin{tikzpicture}[line cap=round,line join=round,>=triangle 45,x=1.0cm,y=1.0cm]
			
			\filldraw[fill=gray!40!white, draw=white] (0,0)--(-.85,.73)--(1.5,2.2)--(0,0);

			\draw[->,color=black] (-2.3,0) -- (3,0);

			
			

			
			\draw [] (0,0) parabola (1.61,2.5921); 
			\draw [] (0,0) parabola (-1.61,2.5921);

			\draw [ color=black] (-.85,.71)--(1.5,2.2); 
			\draw [ color=black, dashed] (-2,0)--(-.85,.71);

			\draw [dashed, color=black] (0,0)--(-.85,.73);
			\draw [dashed, color=black] (0,0)--(1.5,2.2);
			
			\draw [color=black] (-.54,.48) -- (.6,.86);
			\draw [color=black, dashed] (-2,0) -- (-.54,.48);
			\draw [color=black, dashed](.6,.86) -- (3, 1.685);


			\draw[->,color=black] (0,0) -- (0,2.8);
			
			\draw (-2,0) node [below] {$p$};
			\draw (-.82,.78) node [left] {$p_u$};
			\draw (1.45,2.2) node [right] {$p_v$};
			\draw (3,1.6) node [above] {$\ell$};
			\draw (0,0) node [below] {$o$};
			\draw (.55,.86) node [above] {\fontsize{8}{2}{$q_2$}};
			\draw (-.4,.43) node [above] {\fontsize{8}{2}{$q_1$}};
			\draw (.2,1.3) node [above] {$\tilde{o}$};
			\draw (-1, .1) node [below] {$p'$};
			
			\begin{scriptsize}
			
			\fill [color=black] (0,1.25) circle (1.1pt);
			\fill [color=black] (-2,0) circle (1.1pt);
			\fill [color=black] (1.5,2.2) circle (1.1pt);
			\fill [color=black] (-.85,.71) circle (1.1pt);
			\fill [color=black] (-.54,.48) circle (1.1pt);
			\fill [color=black] (.6,.86) circle (1.1pt);
			\fill [color=black] (0,0) circle (1.1pt);
			\fill [color=black] (-1,0) circle (1.1pt);
			
			\end{scriptsize}
			
			\end{tikzpicture}
			
			\caption{The first wall $\mathcal{W}_{i_*F}$}
			
			\label{wall.43}
			
		\end{centering}
		
	\end{figure}

	The wall $\mathcal{W}_{i_*F}$ passes through $\sigma_{(0,w')}$ for some $w' >0$. There is a destabilising sequence $F_1 \hookrightarrow i_*F \twoheadrightarrow F_2$ of objects in $\mathcal{A}(0)$ such that $F_1$ and $F_2$ are $\sigma_{(0,w')}$-semistable objects of the same phase as $i_*F$ and $\phi_{(0,w)}(F_1) > \phi_{(0,w)}(i_*F)$ for $w < w'$. We may assume $F_1$ is $\sigma_{(0,w')}$-stable. Taking cohomology gives a long exact sequence of sheaves 
	\begin{equation}\label{les.2}
	0 \rightarrow H^{-1}(F_2) \rightarrow F_1 \xrightarrow{d_0}  i_*F \rightarrow H^0(F_2) \rightarrow 0.
	\end{equation}
	Let $v(F_1) = \big(r',c'H,s'\big)$ and $v\big(H^0(F_2)\big) = \big(0,c''H,s''\big)$. 
	Let $T$ be the maximal torsion subsheaf of $F_1$. Using \cite[Lemma 3.5]{feyz:mukai-program} in the same way as in the proof of \cite[Proposition 4.2]{feyz:mukai-program} gives
	$0 < r'= m^2-c''- c(T)$ and 
	\begin{equation}\label{slope of torsion-free}
	\mu_H\big(F_1/T\big) = \dfrac{c'-c(T)}{m^2-c''-c(T)} = \dfrac{1}{m}\ . 
	\end{equation}
	Moreover, it shows both sheaves $F_1/T$ and $H^{-1}(F_2)$ are $\mu_H$-semistable.
	We first prove the torsion part $T$ is equal to zero. Assume otherwise. 
	
	\textbf{Step 1.} By the definition \eqref{heart} of the heart $\mathcal{A}(b= 0)$, the sequence $T \hookrightarrow F_1 \twoheadrightarrow F_1/T$ is a short exact sequence in $\mathcal{A}(0)$. Combining this with the desalinating sequence gives an injection $f \colon T \hookrightarrow i_*F$ in $\mathcal{A}(0)$. By taking cohomology from the sequence $T \hookrightarrow i_*F \twoheadrightarrow \text{cok} (f)$, we get a long exact sequence of coherent sheaves 
	\begin{equation*}
	0 \rightarrow H^{-1}(\text{cok}(f)) \rightarrow T \rightarrow i_*F \rightarrow H^0(\text{cok}(f)) \rightarrow 0\ . 
	\end{equation*}  
	Since $H^{-1}(\text{cok}(f))$ is torsion-free by definition of the heart $\mathcal{A}(0)$, it must be equal to zero. Thus $T$ is a subsheaf of $i_*F$, so slope semistability of $F$ implies
	\begin{equation}\label{slope}
	c(T) \neq 0\  \qquad \text{and} \qquad \frac{s(T)}{c(T)} \leq \frac{s(i_*F)}{c(i_*F)} < 0\ .
	\end{equation}
	Therefore $p' \coloneqq pr(v(T)) = \left(\frac{c(T)}{s(T)} , 0 \right) $ lies on the line segment $[\overline{po})$ where $p \coloneqq pr(v(i_*F))$ in Figure \ref{wall.43}. 
	
	\textbf{Step 2.} The next step is to show $pr(v(F_1/T))$ lies on the line segment $(\overline{oq_2}]$. By the equality \eqref{slope of torsion-free}, $pr(v(F_1/T))$ lies on the line passing through the origin $o$ and $p_v$. The short exact sequence $T \hookrightarrow F_1 \twoheadrightarrow F_1/T$ implies that $p'$, $pr(v(F_1))$ and $pr(v(F_1/T))$ are collinear. Moreover, $pr(v(F_1))$ lies on the line $\ell$.

	If $s(F_1/T) = 0$, then $s(F_1) = s(T) <  0$ by \eqref{slope}, and the line passing through $pr(v(F_1))$ and $p'=pr(v(T))$ is of slope $1/\mu(F_1/T)$, so it is parallel to $\overline{op_v}$. The point $pr(v(F_1))$ also lies on $\ell$, thus $pr(v(F_1))$ is on the line segment $[\overline{pq_2}]$. But this is not possible because $\mu_H(F_1) \geq \frac{1}{m}$ by \eqref{slope of torsion-free}. Hence we may assume $s(F_1/T) \neq 0$. Consider the following three cases: 
	\begin{enumerate*}
		\item If $s' < 0$, then since $r'>0$ and $c' >0$, $pr(v(F_1))$ lies on $\ell$ on the left hand side of $p$, thus $\frac{c'}{s'} - \frac{c(T)}{s(T)} \leq \frac{c'}{s'} - \frac{c(i_*F)}{s(i_*F)} < 0$ by \eqref{slope} which gives
		\begin{equation*}
		0 < \frac{\frac{r'}{s'}}{\frac{c'}{s'} - \frac{c(T)}{s(T)}} \leq \frac{\frac{r'}{s'}}{\frac{c'}{s'} - \frac{c(i_*F)}{s(i_*F)}}\ . 
		\end{equation*} 
		This implies $pr(v(F_1/T))$ lies on a line passing through $p'$ and of positive slope but less than slope of $\ell$, hence it lies on the line segment $(\overline{oq_2}]$. 
		\item If $s' >0$, then $pr(v(F_1))$ lies on $\ell$ on the right hand side of $q_2$ because $\sigma_{0,w'}$-stability of $F_1$ gives $v(F_1)^2 \geq -2$ and there is no root class on $(\overline{q_1q_2})$. 
		Hence, the line passing through $pr(v(F_1))$ and $q'$ intersects $\overline{op_v}$ at a point on $(\overline{oq_2}]$.
		\item If $s' = 0$, then $pr(v(F_1/T))$ lies on a line parallel to $\ell$ (i.e. of slope $\frac{r'}{c'}$) which passes through $p'$, hence it lies on the line segment $(\overline{oq_2}]$.   
	\end{enumerate*}  
	
	\textbf{Step 3.} We show $F_1/T$ is $H$-Gieseker stable. If not, since it is $\mu_H$-semistable there is an $H$-Gieseker stable subsheaf $F'$ of the same slope such that 
	\begin{equation*}
	0 < \frac{s(F_1/T)}{r(F_1/T)} \leq \frac{s(F')}{r(F')} \;\;\; \Rightarrow \;\;\; \frac{r(F')}{s(F')} \leq \frac{r(F_1/T)}{s(F_1/T)}
	\end{equation*}
	which implies $pr(v(F'))$ lies on $(\overline{oq_2}]$. Since $F'$ is $H$-Gieseker stable, $v(F')^2 \geq -2$ but Lemma \ref{lem. no psheircal in grey region} shows that there is no projection of roots on $(\overline{op_v})$, a contradiction. Therefore, $F_1/T$ is $H$-Gieseker stable and $v(F_1/T)^2 \geq -2$. Since there is no projection of root class on the line segment $(\overline{op_u})$, we obtain $pr(v(F_1/T)) = p_v = q_2$. By considering their second coordinates, one gets 
	\begin{equation}\label{slope conidiotn}
	\frac{c' - c(T)}{s'-s(T)} = \frac{m}{p}.
	\end{equation}
	We know $\text{Im}[Z_{(0,w')}(F_1)] = c' \leq  \text{Im}[Z_{(0,w')}(i_*F)] = m$ and and $\gcd(m, p)=1$, thus \eqref{slope conidiotn} gives $c'- c(T) = m$. Then \eqref{slope of torsion-free} implies $m^2 - c'' - c(T) = m^2$, i.e. $c'' = c(T)= 0$ which is not possible by \eqref{slope}. 
	
	Therefore $T = 0$ and $F_1$ is a $\mu_H$-semistable sheaf of slope
	\begin{equation*}
	\mu_H(F_1) = \frac{c'}{m^2 - c''} = \frac{1}{m}\ . 
	\end{equation*}     
	If $s' = 0$, then $\ell$ is of slope $\frac{r'}{c'}$ (parallel to $\overline{op_v}$) and passes through $p$. Thus $\ell$ lies above $\overline{p_up_v}$ which is not possible by our assumption, so we may assume $s' \neq 0$. We know $pr(v(F_1))$ is the intersection point of $\ell$ and the line passing through $o$ and $p_v$. Therefore $pr(v(F_1)) = q_2$ but the same argument as above shows that $F_1$ is $H$-Gieseker stable and $pr(v(F_1)) = q_2 = p_v$. Thus there is no wall for $i_*F$ below $\overline{p_up_v}$.  
	
	If $\ell$ passes through $p_v$, then $pr(v(F_1)) = p_v$, so $\frac{c'}{s'} = \frac{m}{p}$ which implies $c' = m$ and $s ' =p$. Thus $v(F_1) = v_1$. The same argument as in Proposition \ref{5.2} implies that there is no wall for $F_1$ in the triangle $\triangle op_up_v$, so $\sigma_{0,w'}$-stability of $F_1$ implies that it is $\sigma_{(0,w)}$-stable for $w \gg 0$, i.e. when $k(0,w)$ is close to the origin $o$. Thus $F_1$ is $H$-Gieseker stable by \cite[Proposition 14.2]{bridgeland:K3-surfaces}. Then Proposition \ref{5.2} implies that it is a $\mu_H$-stable locally free sheaf and $F_1|_C$ is slope stable. The non-zero morphism $d_0$ in the long exact sequence \eqref{les.2} factors via the morphism $d_0' \colon i_*F_1|_C \rightarrow i_*F$. The objects $i_*F_1|_C$ and $i_*F$ have the same Mukai vector and so have the same phase and $i_*F_1|_C$ is $H$-Gieseker stable. Thus, the morphism $d_0'$ must be an isomorphism and $F \cong F_1|_C$.
\end{proof}
\subsection{The maximum number of global sections}
Take a slope semistable vector bundle $F \in M_C(m^2,2pm)$. Proposition \ref{polygon} implies that there is a positive real number $w^* >0$ such that the HN filtration of $i_*F$ with respect to the stability conditions $\sigma_{(0,w)}$ for $\sqrt{1/p} < w <w^*$ is a fixed sequence 
\begin{equation*}
0=\tilde{E}_0 \subset \tilde{E}_1 \subset .... \subset \tilde{E}_{n-1} \subset \tilde{E}_n=i_*F\ . 
\end{equation*} 
Let 
\begin{equation*}
p_i \coloneqq \overline{Z}(\tilde{E}_i)  = \text{rk}(\tilde{E}_i) - \text{s}(\tilde{E}_i) + i \text{c}(\tilde{E}_i) 
\end{equation*}
for $0 \leq i \leq n$, and $P_{i_*F}$ be a polygon whose sides are $\overline{p_np_0}$ and $\overline{p_ip_{i+1}}$ for $i=0, \dots, n-1$. By Proposition \ref{polygon}, $P_{i_*F}$ is a convex polygon and 
\begin{equation}\label{in.polygon}
h^0(X,i_*F) \leq \dfrac{\chi(i_*F)}{2}  + \dfrac{1}{2} \sum_{i=0}^{n-1} \lVert \overline{p_ip_{i+1}} \rVert,
\end{equation}
where $\lVert . \rVert$ is the non-standard norm on $\mathbb{C}$:
\begin{equation*}
\lVert x+iy \rVert  =  \sqrt{x^2 + (4p+4)y^2}.
\end{equation*}      

By applying the same argument as in \cite[Lemma 4.3]{feyz:mukai-program}, one can show that the polygon $P_{i_*F}$ is inside the triangle $\triangle z_1z_2o$ with vertices 
\begin{equation*}
z_1 \coloneqq \overline{Z}(v) = m^2-p + i\,m, \;\;\;\; z_2 \coloneqq \overline{Z}(i_*F) = m^2p-2pm + i\, m^2
\end{equation*}
and the origin $o$. Now assume $F$ is in the Brill-Noether locus $M_{C}(m^2, 2pm, p+m^2)$, i.e.  
$$
p+m^2 \leq h^0(X, i_*F). 
$$
We show in the following few lemmas that the polygon $P_{i_*F}$ coincides with the triangle $\triangle oz_1z_2$.  

\begin{Lem}\label{lem. large}
	If $p \geq 31$ and $p \neq 47, 59$, the polygon $P_{i_*F}$ coincides with the triangle $\triangle oz_1z_2$. 
\end{Lem} 
\begin{proof}
	Assume for a contradiction that $P_{i_*F}$ is strictly inside $\triangle oz_1z_2$. Since the vertices of $P_{i_*F}$ are Gaussian integers, it must be contained in the polygon $oz_1'qz_2'z_2$ where 
	\begin{equation*}
	z_1' \coloneqq \frac{m-1}{m}(m^2 -p) +i \, (m-1), \;\;\;\; z_2' \coloneqq -\frac{p}{m} +m^2-\frac{m}{m-1} + i\, (m+1)
	\end{equation*}
	and $q \coloneqq m^2-p+1 +i\, m$, see Figure \ref{fig. polygon}.
	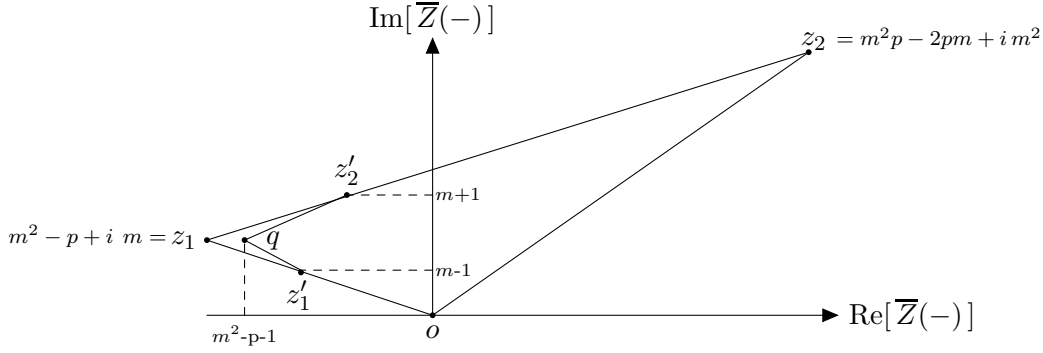
\begin{figure}[h]
		\begin{centering}
			\definecolor{zzttqq}{rgb}{0.27,0.27,0.27}
			\definecolor{qqqqff}{rgb}{0.33,0.33,0.33}
			\definecolor{uququq}{rgb}{0.25,0.25,0.25}
			\definecolor{xdxdff}{rgb}{0.66,0.66,0.66}
			
			\begin{tikzpicture}[line cap=round,line join=round,>=triangle 45,x=1.0cm,y=1.0cm]

			\draw[->,color=black] (0,0) -- (0,3.7);
			\draw[->,color=black] (-3,0) -- (5.4,0);
			
			\draw[color=black] (0,0) -- (5,3.5);
			\draw[color=black] (0,0) -- (-3,1);
			\draw[color=black] (-3,1) -- (5,3.5);
			
			\draw[color=black] (-2.5 ,1) -- (-1.75, .6);
			\draw[color=black, dashed] (0 ,.6) -- (-1.75, .6);
			\draw[color=black, dashed] (-2.5, 0) -- (-2.5,1);
			
			\draw[color=black] (-2.5,1) -- (-1.14,1.6);
			\draw[color=black, dashed] (0,1.6) -- (-1.14,1.6);
			
			
			\draw (0,0) node [below] {$o$};
			\draw (5.4,0) node [right] {Re$[\,\overline{Z}(-)\,]$};
			\draw (0,3.6) node [above] {Im$[\,\overline{Z}(-)\,]$};
			\draw (-3.4,1.05) node [left] {\fontsize{8}{2}{$ m^2-p+i \ m =$}};
			\draw (-3,1) node [left] {$z_1$};
			\draw (-1.78,.6) node [below] {$z_1'$};
			\draw (6.5,3.4) node [above] {$z_2$ \fontsize{8}{2}{$ = m^2p-2pm+i\,m^2$}};
			\draw (-1.14,1.6) node [above] {$z_2'$};
			\draw (-2.35,1) node [right] {$q$};
			
			\draw (-.1,1.6) node [right] {\fontsize{7}{2}{$m\text{+1}$}};
			\draw (-.1,.6) node [right] {\fontsize{7}{2}{$m\text{-1}$}};
			\draw (-2.5,0) node [below] {\fontsize{7}{2}{$m^2\text{-p-1}$}};
			
			\begin{scriptsize}
			
			\fill [color=black] (0,0) circle (1.1pt);
			\fill [color=black] (5,3.5) circle (1.1pt);
			\fill [color=black] (-3,1) circle (1.1pt);
			\fill [color=black] (-1.75,.57) circle (1.1pt);
			\fill [color=black] (-1.14,1.6) circle (1.1pt);
			\fill [color=black] (-2.5,1) circle (1.1pt);
			
			\end{scriptsize}
			
			\end{tikzpicture}
			
			\caption{The polygon $P_{i_*F}$ is inside $oz_1'qz_2'z_2$}
			
			\label{fig. polygon}
			
		\end{centering}
		
	\end{figure}
	
	Let $l \coloneqq \lVert \overline{oz_1} \rVert + \lVert \overline{z_1z_2} \rVert$ and $l_{in} \coloneqq \lVert \overline{oz_1'} \rVert + \lVert \overline{z_1'q} \rVert+ \lVert \overline{qz_2'} \rVert + \lVert \overline{z_2'z_2} \rVert $. A similar argument as in \cite[Proposition 4.4]{feyz:mukai-program} implies that $	l -l_{in} \leq 2\epsilon$ where $\epsilon \coloneqq 
	\frac{2pm-m^2p}{2} + \frac{l}{2} - (p+m^2)$. However, we show that if $p\geq 31$ and $p \neq 47, 59$, then 
	\begin{equation}\label{claim for contardiction}
	2\epsilon < l-l_{in}
	\end{equation}
	which is a contradiction.

	\textbf{Step 1.} We have
	\begin{align*}
	l &= \sqrt{(m^2-p)^2 + (4p+4)m^2} + \sqrt{(pm^2 -2pm +p -m^2)^2 + (4p+4)(m^2-m)^2  }\\
	&= \sqrt{(m^2 +p)^2 + 4m^2} + \sqrt{(p(m-1)^2 +m^2 )^2 + 4m^2(m-1)^2 }\ , 
	\end{align*}
	thus
	\begin{align}
	2 \epsilon = & \,l + p(-m^2 + 2m) -2(p+m^2)\nonumber \\
	=&\, \sqrt{(m^2 +p)^2 + 4m^2} - (m^2+p) \nonumber \\
	&+ \sqrt{(p(m-1)^2 +m^2 )^2 + 4m^2(m-1)^2 } - (p(m-1)^2 +m^2) \label{2epsilcon} \\
	\leq &\  \frac{2m^2}{m^2+p} + \frac{2m^2(m-1)^2}{p(m-1)^2+m^2} < \frac{2m^2}{m^2+p} + \frac{2m^2}{p+1}\eqqcolon f_1(m , p)\ .\label{2epsilon-2} 
	\end{align}
	On the other hand, $l -l_{in} = \lVert \overline{z_1'z_1} \rVert - \lVert \overline{z_1'q} \rVert + \lVert \overline{z_1z_2'} \rVert - \lVert \overline{qz_2'} \rVert $ where 
	\begin{align}\label{in. first part}
	\lVert \overline{z_1'z_1} \rVert - \lVert \overline{z_1'q} \rVert  &= \sqrt{\left(\frac{p-m^2}{m} \right)^2 + 4p+4  } - \sqrt{\left(  \frac{p-m^2}{m} -1\right)^2 + 4p+4  }\nonumber\\
	& = \frac{ - 1 +  \frac{2p}{m} -2m  }{\sqrt{\left(\frac{p-m^2}{m} \right)^2 + 4p+4  } + \sqrt{\left(  \frac{p-m^2}{m} -1\right)^2 + 4p+4  }}\ .
	\end{align}
	If $\frac{p-m^2}{m} > 1$, then
	\begin{equation}\label{in. first part-2}
	\lVert \overline{z_1'z_1} \rVert - \lVert \overline{z_1'q} \rVert \geq \frac{\frac{-1}{2} + \frac{p}{m} -m  }{\sqrt{\left(  \frac{p}{m} +m\right)^2 + 4  }  } \eqqcolon f_2(m, p)\ .
	\end{equation}
	We also have 
	\begin{align}
	\lVert \overline{z_1z_2'} \rVert - \lVert \overline{qz_2'} \rVert =&  \;\sqrt{\left(p \frac{m-1}{m} - \frac{m}{m-1} \right)^2 + 4p+4  }\nonumber \\ &\;\;\;\;\;- \sqrt{\left(p \frac{m-1}{m} - \frac{m}{m-1} -1\right)^2 + 4p+4  }\, \label{in. second part}\\
	\geq &\; \frac{\frac{-1}{2} + p \frac{m-1}{m} - \frac{m}{m-1}   }{\sqrt{\left(p \frac{m-1}{m} + \frac{m}{m-1} \right)^2 + 4  }} > \frac{p \frac{m-1}{m} - \frac{3m-1}{2(m-1)} }{p \frac{m-1}{m} + \frac{3m-1}{2(m-1)}} \eqqcolon f_3(m, p) \ .\label{in. second part-2}
	\end{align} 
	
	\textbf{Step 2.} First consider prime numbers $p$ such that $31 \leq p < 250$ and $p \neq 47, 59$. 
\begin{enumerate}
		\item If $p =31, 37, 43, 61, 67, 73, 79, 97, 103, 109, 127, 139, 151, 157, 163, 181, 193, 199, 211,$ $ 223, 233$ or $241$, then $m=3$ so \eqref{2epsilon-2} gives 
		\begin{equation}\label{f.1}
		2\epsilon \leq f_1(3, p) = \frac{18}{9+p} + \frac{18}{p+1} \leq \frac{81}{80}\ . 
		\end{equation}
		Since $\frac{p-9}{3} > 1$, \eqref{in. first part-2} and \eqref{in. second part-2} give  
		\begin{equation*}
		l -l_{in } \geq f_2(3, p) + f_3(3, p) = \frac{\frac{-1}{2} + \frac{p}{3} -3  }{\sqrt{\left(  \frac{p}{3} +3\right)^2 + 4  }  } + \frac{\frac{2p}{3} - 2 }{\frac{2p}{3} + 2}\ .
		\end{equation*}
		Since the function on the right hand side is increasing with respect to $p$, we get $l-l_{in} > \frac{13}{10}$ because $p \geq 31$. Combining this with \eqref{f.1} implies the claim \eqref{claim for contardiction}. 
		\item If $p = 41, 53, 89, 101, 113, 137, 149, 173$ or $229$, then $m=4$. For $p=41$, a direct computation shows \eqref{claim for contardiction}. For $p \geq 53$ we have $2\epsilon \leq f_1(4, p) \leq \frac{11}{10}$ and 
		\begin{equation*}
		l -l_{in } \geq f_2(4, p) + f_3(4, p) = \frac{\frac{-1}{2} + \frac{p}{4} -4  }{\sqrt{\left(  \frac{p}{4} +4\right)^2 + 4  }  } + \frac{\frac{3p}{4} - \frac{11}{6} }{\frac{3p}{4} + \frac{11}{6}} > \frac{14}{10}
		\end{equation*}
		which proves \eqref{claim for contardiction}. 
		\item If $p=71, 83, 107, 131, 167, 191, 197$ or $227$, then $m=5$. Thus $2 \epsilon \leq f_1(5, p) \leq \frac{13}{10}$ and $l - l_{in} \geq f_2(5, p)+ f_3(5, p) > \frac{14}{10}$, so \eqref{claim for contardiction} holds.   
		\item If $p= 179$ or $239$, then $m=7$. So $2\epsilon \leq f_1(7, p) < 1$ and $l - l_{in} > f_2(7, p)+ f_3(7, p) > \frac{15}{10}$ which proves \eqref{claim for contardiction}. 
	\end{enumerate}

	\textbf{Step 3.} If $p > 250$, then we claim that  $m \leq \left\lfloor \sqrt{\frac{2p}{5}} \right\rfloor$. If not, three consecutive numbers $\left\lfloor \sqrt{\frac{2p}{5}} \right\rfloor$, $\left\lfloor \sqrt{\frac{2p}{5}} \right\rfloor -1$ and $\left\lfloor \sqrt{\frac{2p}{5}} \right\rfloor-2$ divide $p+1$, thus 
	\begin{align*}
	p+1 &\geq \frac{1}{2}\left\lfloor \sqrt{\frac{2p}{5}} \right\rfloor\left(\left\lfloor \sqrt{\frac{2p}{5}} \right\rfloor -1\right)\left(\left\lfloor \sqrt{\frac{2p}{5}} \right\rfloor -2\right)\\
	& \geq \frac{1}{2}\left(\sqrt{\frac{2p}{5}}  -1\right) \left(\sqrt{\frac{2p}{5}}  -2\right)\left(\sqrt{\frac{2p}{5}}  -3\right)
	\end{align*}
	which does not hold for $p >  250$. The function $f_1(m, p)$ in \eqref{2epsilon-2} is increasing with respect to $m$. Since $m \leq \left\lfloor \sqrt{\frac{2p}{5}} \right \rfloor \leq \sqrt{\frac{2p}{5}}\,$, we get
	$f_1(m, p) < \frac{4}{7} + \frac{4}{5 }$, thus 
	\begin{equation}\label{epsilon}
	2 \epsilon < \frac{48}{35}\,.
	\end{equation}
	On the other hand, the function $f_2(m, p)$ in \eqref{in. first part-2} is decreasing with respect to $m$, thus 
	\begin{equation*}
	f_2(m, p) \geq  \frac{-\frac{1}{2} + \sqrt{\frac{5p}{2}} - \sqrt{\frac{2p}{5}}}{\sqrt{ \left(\sqrt{\frac{5p}{2}} + \sqrt{\frac{2p}{5}}\right)^2 +4   }} \ .
	\end{equation*}
	The function on the right hand side is increasing with respect to $p$, so takes its minimum value for $p=251$. Therefore $f_2(m, p) > \frac{517}{1250}$. 
	Since $m \geq 3$, we have $	\frac{1.5m -.5}{m-1}. \frac{m}{m-1} \leq 3\ $. Therefore $p > \frac{3m -1}{2(m-1)}. \frac{m}{m-1} \frac{1975}{25}$ which is equivalent to  $f_3(m, p) > \frac{975}{1000}$ in \eqref{in. second part-2}. Hence 
	$
	l- l_{in} \geq f_2(m, p) + f_3(m, p) \geq \frac{13886}{10000}
	$
	and by \eqref{epsilon}, $	l- l_{in} > 2 \epsilon$ as claimed.  
\end{proof}

\begin{Lem} \label{lem. 13}
	If $(p, m) = (13, 3)$, the polygon $P_{i_*F}$ coincides with the triangle $ \triangle oz_1z_2$. 
\end{Lem} 

\begin{proof}
	
	In this case, $z_1 = -4+ i3$, $z_2 = 39 + i9$, $z_2' = \frac{19}{6} +i4$ and $22 \leq h^0(F)$.
	
	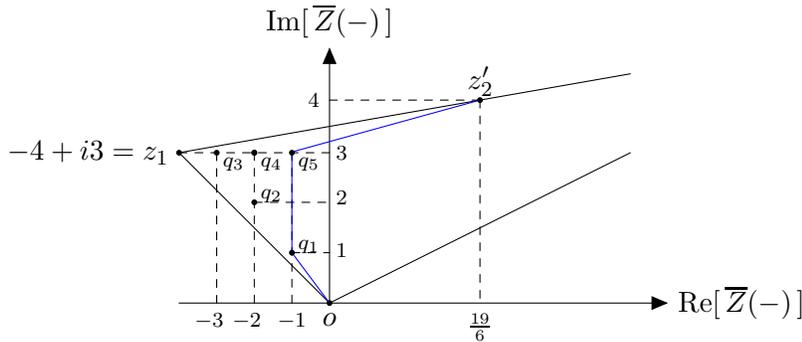
\begin{figure}[h]
		\begin{centering}
			\definecolor{zzttqq}{rgb}{0.27,0.27,0.27}
			\definecolor{qqqqff}{rgb}{0.33,0.33,0.33}
			\definecolor{uququq}{rgb}{0.25,0.25,0.25}
			\definecolor{xdxdff}{rgb}{0.66,0.66,0.66}
			
			\begin{tikzpicture}[line cap=round,line join=round,>=triangle 45,x=1.0cm,y=1.0cm]

			\draw[->,color=black] (0,0) -- (0,3.4);
			\draw[->,color=black] (-2,0) -- (4.5,0);
			\draw[color=black] (0,0) -- (-2,2);
			\draw[color=black] (-2, 2) -- (4,3.05);
			\draw[color=black] (0,0) -- (4,2);
			
			\draw[color=black, dashed] (-2,2) -- (0, 2);
			\draw[color=black, dashed] (2,2.7) -- (2, 0);
			\draw[color=black, dashed] (2,2.7) -- (0, 2.7);
			\draw[color=black, dashed] (-.5,0) -- (-.5, 2);
			\draw[color=black, dashed] (-1,0) -- (-1, 2);
			\draw[color=black, dashed] (-.5,.67)  -- (0,.67);
			\draw[color=black, dashed] (-1,1.34)  -- (0,1.34);
			\draw[color=black, dashed] (-1.5,2)  -- (-1.5, 0);
			
			\draw[color=blue] (0, 0)  -- (-.5, .67);
			\draw[color=blue] (-.5, 2.01)  -- (-.5, .67);
			\draw[color=blue] (-.5, 2.01)  --(2,2.7);
			%
			\draw (4.5,0) node [right] {Re$[\,\overline{Z}(-)\,]$};
			\draw (0,3.4) node [above] {Im$[\,\overline{Z}(-)\,]$};
			\draw (-2,2) node [left] {$-4+ i3 = z_1$};
			\draw (1.7, 2.95) node [right] {$z_2'$};
			%
			\draw (-1.05,1.43) node [right] {\fontsize{8}{2}{$q_2$}};
			\draw (-.55,.75)  node [right] {\fontsize{8}{2}{$q_1$}};
			\draw (-1.55,1.85)  node [right] {\fontsize{8}{2}{$q_3$}};
			\draw (-1.05,1.85)  node [right] {\fontsize{8}{2}{$q_4$}};
			\draw (-0.55,1.85)  node [right] {\fontsize{8}{2}{$q_5$}};
			\draw (-.05,1.4) node [right] {\fontsize{8}{2}{$2$}};
			\draw (-.05,.7)  node [right] {\fontsize{8}{2}{$1$}};
			\draw (-.05,2.)  node [right] {\fontsize{8}{2}{$3$}};
			\draw (0,2.7)  node [left] {\fontsize{8}{2}{$4$}};
			
			\draw (-.5,0)  node [below] {\fontsize{8}{2}{$-1$}};
			\draw (-1.1, 0)  node [below] {\fontsize{8}{2}{$-2$}};
			\draw (-1.6, 0)  node [below] {\fontsize{8}{2}{$-3$}};
			\draw (2, 0)  node [below] {\fontsize{8}{2}{$\frac{19}{6}$}};
			
			\draw (0,0) node [below] {$o$};

			\begin{scriptsize}
			
			\fill [color=black] (0,0) circle (1.1pt);
			\fill [color=black] (2,2.7) circle (1.1pt);
			\fill [color=black]  (-2,2) circle (1.1pt);
			\fill [color=black] (-1.5,2) circle (1.1pt);
			\fill [color=black] (-1,1.34) circle (1.1pt);
			\fill [color=black] (-.5,.67) circle (1.1pt);
			\fill [color=black] (-.5, 2) circle (1.1pt);
			\fill [color=black] (-1, 2) circle (1.1pt);
			\fill [color=black] (-1, 2) circle (1.1pt);

			\end{scriptsize}
			
			\end{tikzpicture}
			
			\caption{The left hand side of $P_{i_*F}$ for $(p,m)=(13, 3)$}
			
			\label{fig. p=13}
			
		\end{centering}
		
	\end{figure}
	
	If $P_{i_*F}$ is contained in the polygon $oq_1q_5z_2'z_2$ in Figure \ref{fig. p=13}, then
	\begin{align*}
	h^0(F) &\leq -\frac{39}{2} + \frac{1}{2} \left(\lVert \overline{oq_1} \rVert + \lVert \overline{q_1q_5} \rVert+ \lVert \overline{q_5z_2'} \rVert+ \lVert \overline{z_2'z_2} \rVert \right)\\
	& = -\frac{39}{2} + \frac{1}{2} \left(\sqrt{56+1} + \sqrt{56\times 4} + \sqrt{56+ \left(\frac{19}{6} +1\right)^2} + \sqrt{56 \times 25+ \left(39-\frac{19}{6} \right)^2}\right)\\
	&< 22
	\end{align*} 
	which is not possible. Thus if $P_{i_*F}$ is strictly inside $\triangle oz_1z_2$, one of the following cases happens: 
	\begin{itemize*}
		\item[i.] $q_2 = -2+i2$ is a vertex of $P_{i_*F}$ and $q_3 = -3+i3$ is not a vertex. We divide $P_{i_*F}$ into two parts. The first part is $oq_2$ and the second part is between $q_2$ and $z_2$. The length of the second part is less than $\lVert \overline{q_2q_4}\rVert + \lVert \overline{q_4z_2'}\rVert+ \lVert \overline{z_2'z_2}\rVert$, see Figure \ref{fig. p=13}. We apply Remark \ref{Rem.polygon} for these two parts. Since $q_2$ has even real part and $z_2$ has odd real part, we get   
		\begin{align*}
		h^0(F) &\leq \frac{-39 -1}{2} + \left\lfloor \frac{\lVert \overline{oq_2}\rVert}{2} \right\rfloor + \left\lfloor  \frac{\lVert \overline{q_2q_4}\rVert + \lVert \overline{q_4z_2'}\rVert+ \lVert \overline{z_2'z_2}\rVert +1}{2} \right\rfloor= -20 + 7+ 34=21\  .  
		\end{align*} 
		\item [ii.] $q_3$ is a vertex of $P_{i_*F}$. Since $q_3$ has odd real part, applying Remark \ref{Rem.polygon} for $\overline{oq_3}$ and the rest of the polygon between $q_3$ and $z_2$ gives
		\begin{align*}
		h^0(F) \leq \frac{-39-1}{2} + \left\lfloor \frac{\lVert \overline{oq_3}\rVert+1}{2} \right\rfloor + \left\lfloor \frac{\lVert \overline{q_3z_2'}\rVert + \lVert \overline{z_2'z_2}\rVert }{2} \right\rfloor= -20 +11 +30 =21\ .
		\end{align*} 
		\item[iii.] $q_4 = -2+ 3i$ is a vertex of $P_{i_*F}$. The length of the first part of $P_{i_*F}$ between $o$ and $q_4$ is at most $\lVert \overline{oq_2}\rVert + \lVert \overline{q_2q_4}\rVert$ and the length of the rest is at most $\lVert \overline{q_4z_2'}\rVert + \lVert \overline{z_2'z_2}\rVert$. Since $q_4$ has even real part, Remark \ref{Rem.polygon} gives 
		\begin{align*}
		h^0(F) \leq \frac{-39-1}{2} + \left\lfloor \frac{\lVert \overline{oq_2}\rVert + \lVert \overline{q_2q_4}\rVert}{2} \right\rfloor + \left\lfloor \frac{\lVert \overline{q_4z_2'}\rVert + \lVert \overline{z_2'z_2}\rVert +1 }{2} \right\rfloor= -20 +11 +30 =21 \ .
		\end{align*}		
	\end{itemize*}
	Therefore $P_{i_*F}$ coincides with $\triangle oz_1z_2$.  
\end{proof}

\begin{Lem}\label{lem.17} 
	If $(p, m) = (17, 4)$, the polygon $P_{i_*F}$ coincides with $\triangle oz_1z_2$.
\end{Lem} 
\begin{proof}
	In this case, $z_1 = -1 + i 4$, $z_2 = 136 + i 16$, $z_2' = \frac{125}{12} + i 5$ and $33 \leq h^0(F)$. We prove in few steps that if $P_{i_*F}$ is strictly inside $\triangle oz_1z_2$, then $h^0(F)$ is less than $33$ which is not clearly possible.

	\begin{figure}[h]
		\begin{centering}
			\definecolor{zzttqq}{rgb}{0.27,0.27,0.27}
			\definecolor{qqqqff}{rgb}{0.33,0.33,0.33}
			\definecolor{uququq}{rgb}{0.25,0.25,0.25}
			\definecolor{xdxdff}{rgb}{0.66,0.66,0.66}
			
			\begin{tikzpicture}[line cap=round,line join=round,>=triangle 45,x=1.0cm,y=1.0cm]

			\draw[->,color=black] (0,0) -- (0,4.5);
			\draw[->,color=black] (-2,0) -- (8.5,0);
			\draw[color=black] (-.6,2) -- (0, 0);
			\draw[color=black] (-.6, 2) -- (8, 2.62);
			\draw[color=black] (0, 0) -- (8, 1.5);
			\draw[color=black, dashed] (0, 2.5) -- (6.25, 2.5);

			\draw[color=blue] (0,1) -- (.6, 1.5);
			\draw[color=blue] (.6,1.5) -- (1.8, 2);
			\draw[color=blue] (1.8, 2) -- (6.25,2.5);
			\draw[color=black, dashed] (.6, 0) -- (.6, 2);
			\draw[color=black, dashed] (1.2,0) -- (1.2, 2);
			\draw[color=black, dashed] (1.8,0) -- (1.8, 2);
			\draw[color=black, dashed] (1.8,2) -- (-.6, 2);
			\draw[color=black, dashed] (0,1.5) -- (.6, 1.5);
			%
			\draw (8.5,0) node [right] {Re$[\,\overline{Z}(-)\,]$};
			\draw (0,4.5) node [above] {Im$[\,\overline{Z}(-)\,]$};
			\draw (-.6,2) node [left] {$-1 + i 4 = z_1$};


			\draw (6.25, 2.5) node [above] {$z_2'$};
			\draw[color=black, dashed] (6.25,2.5)  -- (6.25,0);	
			\draw (6.25, 0) node [below] {$\frac{125}{12}$};		
			
			\draw (-.05,.42)  node [right] {\fontsize{8}{2}{$q_1$}};
			\draw (-.05,.87)  node [right] {\fontsize{8}{2}{$q_2$}};
			\draw (-.05,1.35)  node [right] {\fontsize{8}{2}{$q_3$}};
			\draw (.07,1.85)  node [left] {\fontsize{8}{2}{$q_4$}};
			\draw (.67,1.85)  node [left] {\fontsize{8}{2}{$q_5$}};
			\draw (1.27,1.85)  node [left] {\fontsize{8}{2}{$q_6$}};
			\draw (1.75,1.85)  node [right] {\fontsize{8}{2}{$q_7$}};
			\draw (.55, 1.35)  node [right] {\fontsize{8}{2}{$q_8$}};
			
			\draw (.6,0)  node [below] {\fontsize{8}{2}{$1$}};
			\draw (1.2, 0)  node [below] {\fontsize{8}{2}{$2$}};
			\draw (1.8, 0)  node [below] {\fontsize{8}{2}{$3$}};
			
			\draw (0, 2.5)  node [left] {\fontsize{8}{2}{$5$}};

			\draw (0,0) node [below] {$o$};

			\begin{scriptsize}
			
			\fill [color=black] (0,0) circle (1.1pt);
			\fill [color=black]  (-.6,2) circle (1.1pt);
			\fill [color=black] (.6,2) circle (1.1pt);
			\fill [color=black] (0, 2) circle (1.1pt);
			\fill [color=black] (.6, 1.5) circle (1.1pt);
			\fill [color=black] (1.2,2) circle (1.1pt);
			\fill [color=black] (1.8,2) circle (1.1pt);
			\fill [color=black] (0, 1.5) circle (1.1pt);
			\fill [color=black] (0, 1) circle (1.1pt);
			\fill [color=black] (0, .5) circle (1.1pt);
			\fill [color=black] (6.25, 2.5) circle (1.1pt);
			\end{scriptsize}
			
			\end{tikzpicture}
			
			\caption{The left hand side of $P_{i_*F}$ for $(p,m) = (17, 4)$.}
			
			\label{fig. p=17}
			
		\end{centering}
		
	\end{figure}
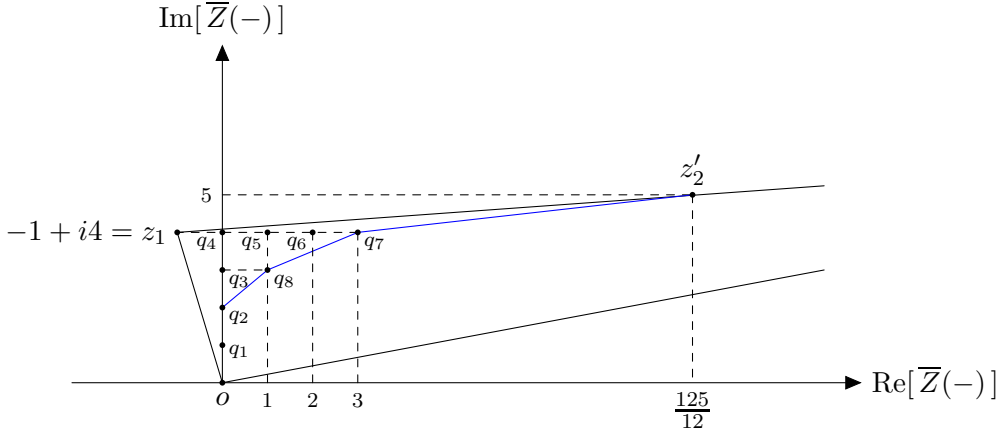
	
	
	For $k=1, 2, 3, 4$ we have $q_k = 0+ i k$ in Figure \ref{fig. p=17}, and for $k =5, 6, 7$, we have $q_{k} = k-4 + i4 \ $.  
	
	\begin{enumerate}
		\item [\textbf{Step 1.}] If $P_{i_*F}$ is contained in the polygon $oq_2q_8q_7z_2'z_2$, then 
		\begin{align*}
		h^0(F) \leq \ &\frac{\chi(F)}{2} + \frac{\lVert \overline{oq_2} \rVert+\lVert \overline{q_2q_8} \rVert+\lVert \overline{q_8q_7} \rVert+\lVert \overline{q_7z_2'} \rVert+\lVert \overline{z_2'z_2} \rVert}{2}\\
		= \ & -\frac{136}{2} + \frac{1}{2}\sqrt{72\times 4 } + \frac{1}{2}\sqrt{72+1}+\frac{1}{2}\sqrt{72+4}+\frac{1}{2}\sqrt{\left(\frac{125}{12} -3\right)^2+72} \\
		\; & +\frac{1}{2}\sqrt{\left(136-\frac{125}{12}\right)^2 + 72 \times 11^2}\\
		< \ & 33 
		\end{align*}
		\item [\textbf{Step 2.}] If $q_6$ is a vertex of $P_{i_*F}$, then in the HN filtration of $i_*F$, there exists a subobject $\tilde{E}_k$ with $\overline{Z}(\tilde{E}_k) = q_6 = 2 + i4$. We first show that 
		\begin{equation}\label{inn.1}
		h^0(\tilde{E}_k) \leq \frac{\chi(\tilde{E}_k)}{2} + 16. 
		\end{equation}  
		\begin{enumerate}
			\item If $\overline{oq_6}$ is a side of $P_{i_*F}$, then either $ \overline{Z}(\tilde{E}_1) = 2 + 4i$, or $\overline{Z}(\tilde{E}_1) = 1 + 2i$ and $\overline{Z}(\tilde{E}_2) = 2 + 4i$. In the first case, Lemma \ref{bound for h} implies $$
			h^0(\tilde{E}_1) \leq \frac{\chi(\tilde{E}_1)}{2} + \left \lfloor \frac{\sqrt{2^2 + 68 \times 4^2 + 4\times 2^2  }}{2} \right \rfloor = \frac{\chi(\tilde{E}_1)}{2} + 16\ .$$ 
			In the second case,  
			\begin{align*}
			h^0(\tilde{E}_2) \leq \ & h^0(\tilde{E}_1) + h^0(\tilde{E}_2/\tilde{E}_1) \\
			\leq \ & \frac{\chi(\tilde{E}_1) -1}{2} + \left \lfloor \frac{\sqrt{1 + 68 \times 2^2 + 4} \ +1}{2} \right \rfloor \\
			& +  \frac{\chi(\tilde{E}_2/\tilde{E}_1) -1}{2} + \left \lfloor \frac{\sqrt{1 + 68 \times 2^2 + 4} \ +1}{2} \right \rfloor
			\end{align*}
			with the last inequality following from the fact that $\chi(\tilde{E}_1)$ and $\chi(\tilde{E}_2/\tilde{E}_1)$ are odd. This gives $h^0(\tilde{E}_2) \leq  \frac{\chi(\tilde{E}_2/\tilde{E})}{2} +15$. 
			\item If $\overline{oq_1}$ and $\overline{q_1q_6}$ are sides of $P_{i_*F}$, then $\overline{Z}(\tilde{E}_1) = i$ and $\overline{Z}(\tilde{E}_2) = 2 + 4i$. Applying Lemma \ref{bound for h} for $\tilde{E}_1$ and $\tilde{E}_2/\tilde{E}_1$ implies that 
			\begin{equation*}
			h^0(\tilde{E}_2) \leq \frac{\chi(\tilde{E}_2)}{2} + \left\lfloor \frac{\sqrt{72}}{2}\right\rfloor + \left\lfloor \frac{\sqrt{2^2 + 68 \times 3^2 + 4}}{2}\right\rfloor = \frac{\chi(\tilde{E}_2)}{2} + 16 \ .
			\end{equation*} 
			\item If $\overline{oq_1}$, $\overline{q_1q_8}$ and $\overline{q_8q_6}$ are sides of $P_{i_*F}$, then $\overline{Z}(\tilde{E}_1) = i$, $\overline{Z}(\tilde{E}_2) = 1 + 3i$ and $\overline{Z}(\tilde{E}_3) = 2 + 4i$. By applying Remark \ref{Rem.polygon} for each of these pieces, one gets 
			\begin{equation*}
			h^0(\tilde{E}_3) \leq \frac{\chi(\tilde{E}_3)}{2} -1 + \left\lfloor \frac{\lVert \overline{oq_1}\rVert}{2} \right\rfloor + \left\lfloor \frac{\lVert \overline{q_1q_8}\rVert +1}{2} \right\rfloor
			+\left\lfloor \frac{\lVert \overline{q_8q_6} \rVert +1}{2} \right\rfloor = \frac{\chi(\tilde{E}_3)}{2}+ 16\ .
			\end{equation*} 
			
			\item If $\overline{oq_2}$ and $\overline{q_2q_6}$ are sides $P_{i_*F}$, then Remark \ref{Rem.polygon} gives 
			\begin{equation*}
			h^0(\tilde{E}_k) \leq \frac{\chi(\tilde{E}_k)}{2} + \left\lfloor \frac{\lVert \overline{oq_2}\rVert}{2} \right\rfloor + \left\lfloor \frac{\lVert \overline{q_2q_6}\rVert }{2} \right\rfloor
			= \frac{\chi(\tilde{E}_k)}{2}+ 16\ .
			\end{equation*} 
			
			\item If $\overline{oq_3}$ and $\overline{q_3q_6}$ are sides of $P_{i_*F}$, then 
			\begin{equation*}
			h^0(\tilde{E}_k) \leq \frac{\chi(\tilde{E}_k)}{2} + \left\lfloor \frac{\lVert \overline{oq_3}\rVert}{2} \right\rfloor + \left\lfloor \frac{\lVert \overline{q_3q_6}\rVert }{2} \right\rfloor
			= \frac{\chi(\tilde{E}_k)}{2}+ 16\ .
			\end{equation*} 	
		\end{enumerate}
		Applying Remark \ref{Rem.polygon} for the rest of the polygon between $q_5$ and $z_2$ implies that   
		\begin{equation*}
		h^0(F) \leq \frac{\chi(F)}{2} + 
		16 +   \left\lfloor \frac{\lVert \overline{q_6z_2'} \rVert+\lVert \overline{z_2'z_2} \rVert}{2} \right\rfloor = -68+ 16+ 84 = 32 \ . 
		\end{equation*}
		
		\item [\textbf{Step 3.}] If $q_5$ is a vertex of $P_{i_*F}$, then there is a subobject $\tilde{E}_k$ in the HN filtration of $i_*F$ with $\overline{Z}(\tilde{E}_k) = 1 +i4$. We first show that 
		\begin{equation}\label{innn.1}
		h^0(\tilde{E}_k) \leq \frac{\chi(\tilde{E}_k) -1}{2} + 17\ . 
		\end{equation}  
		\begin{enumerate}
			\item If $\overline{oq_5}$ is a side of $P_{i_*F}$, then Lemma \ref{bound for h} gives $ h^0(\tilde{E}_k) \leq \frac{\chi(\tilde{E}_k) -1}{2} + \left\lfloor \frac{\sqrt{1 + 68 \times 16 +4}}{2} \right\rfloor $ as claimed.
			\item If $\overline{oq_i}$ and $\overline{q_iq_5}$ are sides of $P_{i_*F}$ for $i =1, 2$ or $ 3$, then  
			$$ h^0(\tilde{E}_k) \leq \frac{\chi(\tilde{E}_k) -1}{2} + \left\lfloor \frac{\lVert \overline{oq_i} \rVert}{2} \right\rfloor + \left\lfloor \frac{\lVert \overline{q_iq_6} \rVert +1}{2} \right\rfloor = \frac{\chi(\tilde{E}_k) -1}{2} + 17$$ 
			as required.  
		\end{enumerate}
		Let $q = \overline{Z}(\tilde{E}_{k+1})$ be the adjacent vertex to $q_5$ which lies between $\overline{z_1z_2}$ and $\overline{q_5z_2}$. Write $q = a+bi$, and let $d = \gcd (a-1 , b-4)$. 
		\begin{enumerate}
			\item  If $a$ is odd, then  
			\begin{equation}\label{in.bound.2}
			h^0(\tilde{E}_{k+1}/\tilde{E}_k) \leq \frac{\chi(\tilde{E}_{k+1}/\tilde{E}_k)}{2} + \ell_1 \;\;\;  \text{and} \;\;\; h^0(i_*F/\tilde{E}_{k+1}) \leq \frac{\chi(i_*F/E_2)-1}{2} + \ell_2
			\end{equation} 
			where $\ell_1 \coloneqq \left\lfloor \frac{\sqrt{4 \times 17 (b-4)^2 + (a-1)^2+4d^2}}{2}\right\rfloor $ and $\ell_2 \coloneqq \left\lfloor \frac{\lVert \overline{qq'} \rVert + \lVert \overline{q'z_2}\rVert}{2} \right\rfloor $ such that $q'$ is a point on the line segment $\overline{z_1z_2}$ whose imaginary part is $b+1$. All possible values of $q$ with odd real part has been listed in Table \ref{table. odd.17}. 
			\begin{table}[ht]
				\centering
				\begin{tabular}{|c|c|c|c|c|c|}
					\hline
					$q$ & $11+i5$ &$23+6i$& $57+i9$ & $79+i11$ & $89+i12$ 	\\
					\hline
					$\ell_1$ &6 &13&35&48&55 \\
					\hline
					$q' $&$\frac{131}{6} + i6$ & $\frac{133}{4} + i7$& $ \frac{135}{2} + i10$ & $\frac{271}{3} + i12$ &$\frac{407}{4} + i13$\\
					\hline 	
					$\ell_2 $&78 &71&49&36&29\\
					\hline 	
				\end{tabular}
				\caption{points $q$ with odd real part} 
				\label{table. odd.17}
			\end{table} 
			
			In all cases $\ell_1 + \ell_2 = 84$. Combining \eqref{innn.1} with \eqref{in.bound.2} gives 
			\begin{equation*}
			h^0(i_*F) \leq h^0(\tilde{E}_k) +h^0(\tilde{E}_{k+1}/\tilde{E}_{k}) + h^0(i_*F/\tilde{E}_{k+1}) \leq \frac{\chi(F)}{2} -1 +17+ \ell_1+\ell_2 =32 \ . 
			\end{equation*}
			
			\item  If $a$ is even, then  
			\begin{equation}\label{in.bound.3}
			h^0(\tilde{E}_{k+1}/\tilde{E}_k) \leq \frac{\chi(\tilde{E}_{k+1}/\tilde{E}_k)-1}{2} + \ell_1 \;\;\;  \text{and} \;\;\; h^0(i_*F/\tilde{E}_{k+1}) \leq \frac{\chi(i_*F/\tilde{E}_{k+1})}{2} + \ell_2
			\end{equation} 
			where $\ell_1 \coloneqq \left\lfloor \frac{\sqrt{4 \times 17 (b-4)^2 + (a-1)^2+4d^2}}{2}\right\rfloor $ and $\ell_2 \coloneqq \left\lfloor \frac{\lVert \overline{qq'} \rVert + \lVert \overline{q'z_2}\rVert}{2} \right\rfloor $ such that $q'$ is a point on the line segment $\overline{z_1z_2}$ whose imaginary part is $b+1$ if $b < 16$. All possible values of $q$ with even real part has been listed in Table \ref{table.even.17}. 
			
			\begin{table}[ht]
				\centering
				\begin{tabular}{|c|c|c|c|c|c|c|c|}
					\hline
					$q$ & $12+i5$ &$22+6i$&$34+i7$&$46+i8$& $68+i10$ & $102+i12$ & $136+i16$ 	\\
					\hline
					$l_1$ &7&13&21&28&42&63&84 \\
					\hline
					$q' $&$\frac{131}{6} + i6$ & $\frac{133}{4} + i7$& $ \frac{134}{3} + i8$ & $\frac{673}{12} + i9$ &$\frac{947}{12}+i11$& $\frac{679}{6}+i14$ &--\\
					\hline 	
					$l_2 $&77&71&63&56&42&21&0\\
					\hline 	
				\end{tabular}
				
				\caption{points $q$ with even first coordinate} 
				\label{table.even.17}
			\end{table} 
			
			In all cases, $\ell_1 + \ell_2 = 84$, thus combining \eqref{innn.1} and \eqref{in.bound.3} gives 
			\begin{equation*}
			h^0(i_*F) \leq h^0(\tilde{E}_k) +h^0(\tilde{E}_{k+1}/\tilde{E}_{k}) + h^0(i_*F/\tilde{E}_{k+1}) \leq  \frac{\chi(F)}{2} -1 +17+ \ell_1+\ell_2 =32\ . 
			\end{equation*}
		\end{enumerate} 	
		
		\item [\textbf{Step 5.}] If $q_4$ is a vertex of $P_{i_*F}$, applying Remark \ref{Rem.polygon} for $\overline{oq_4}$ and the rest of $P_{i_*F}$ gives 
		\begin{equation*}
		h^0(F) \leq \frac{\chi(F)}{2} + 
		\left\lfloor \frac{\lVert \overline{oq_4} \rVert}{2} \right\rfloor +   \left\lfloor \frac{\lVert \overline{q_4z_2'} \rVert+\lVert \overline{z_2'z_2} \rVert}{2} \right\rfloor = -68+ 16+ 84 = 32 \ . 
		\end{equation*}
		
		\item [\textbf{Step 6.}] Suppose $q_3$ is a vertex of $P_{i_*F}$ and none of the above cases holds. Then the length of $P_{i_*F}$ between $q_3$ and $z_2$ is less than $\lVert \overline{q_3q_7'} \rVert +\lVert \overline{q_7'z_2'} \rVert + \lVert \overline{z_2'z_2} \rVert$ where $q_7' = 4+i4$. Then
		\begin{equation*}
		h^0(F) \leq \frac{\chi(F)}{2} + 
		\left\lfloor \frac{\lVert \overline{oq_3} \rVert}{2} \right\rfloor +   \left\lfloor \frac{\lVert \overline{q_3q_7'} \rVert +\lVert \overline{q_7'z_2'} \rVert + \lVert \overline{z_2'z_2} \rVert}{2} \right\rfloor = -68+ 12+ 88 = 32 \ . 
		\end{equation*} 
		
	\end{enumerate}
	
\end{proof}

\begin{Lem}\label{lem.19}
	The polygon $P_{i_*F}$ coincides with the triangle $\triangle oz_1z_2$ when $(p,m) = (19, 3),$ $ (29,4), \, (47, 5)$, or $(59, 7)$.  
\end{Lem}
\begin{proof}
	Let $q_1 \coloneqq m^2-p+1 + i\, m$ and $q_2 \coloneqq m^2-p+2 + i \, m$. Also as before $z_1' \coloneqq \frac{m-1}{m}(m^2 -p) +i \, (m-1)$ and $z_2' \coloneqq -\frac{p}{m} +m^2-\frac{m}{m-1} + i\, (m+1)$.

	\begin{figure}[h]
		\begin{centering}
			\definecolor{zzttqq}{rgb}{0.27,0.27,0.27}
			\definecolor{qqqqff}{rgb}{0.33,0.33,0.33}
			\definecolor{uququq}{rgb}{0.25,0.25,0.25}
			\definecolor{xdxdff}{rgb}{0.66,0.66,0.66}
			
			\begin{tikzpicture}[line cap=round,line join=round,>=triangle 45,x=1.0cm,y=1.0cm]

			\draw[->,color=black] (0,0) -- (0,4);
			\draw[->,color=black] (-2.5,0) -- (7,0);
			\draw[color=black] (-2,1.5) -- (0, 0);
			\draw[color=black] (-2, 1.5) -- (6,4);

			%
			\draw[color=black, dashed] (-4/3, 1) -- (0, 1);
			\draw[color=black, dashed] (-2,1.5) -- (0,1.5);
			\draw[color=black, dashed] (-2/5,2) -- (0, 2);
			

			\draw (7,0) node [right] {Re$[\,\overline{Z}(-)\,]$};
			\draw (0,4) node [above] {Im$[\,\overline{Z}(-)\,]$};
			
			\draw (-2.4,1.5) node [left] {\fontsize{8}{2}{$ m^2-p+i \ m =$}};
			\draw (-2,1.5) node [left] {$z_1$};
			
			\draw (-1.75,1.5) node [above] {\fontsize{8}{2}{$q_1$ }};
			\draw (-1.4,1.6) node [above] {\fontsize{8}{2}{$q_2$ }};
			
			\draw (7.5, 4) node [above] {$z_2$ \fontsize{8}{2}{$ = m^2p-2pm+i\,m^2$}};
			
			\draw (-.1, 1) node [right] {\fontsize{7}{2}{$m-1$}};
			\draw (-.1, 2) node [right] {\fontsize{7}{2}{$m+1$}};
			\draw (-.1, 1.5) node [right] {\fontsize{7}{2}{$m$}};
			
			\draw (-2/5, 2)  node [above] {\fontsize{8}{2}{$z_2'$}};
			\draw (-1.2 , .9)  node [above] {\fontsize{8}{2}{$z_1'$}};
			\draw (0,0) node [below] {$o$};

			\draw[color=blue] (-4/3,1) -- (-1.5,1.5);
			\draw[color=blue] (-2/5, 2) -- (-1.5,1.5);
			\draw[color=black] (0,0) -- (6,4);
			
			\begin{scriptsize}
			
			\fill [color=black] (0,0) circle (1.1pt);
			
			\fill [color=black] (6, 4) circle (1.1pt);
			
			\fill [color=black] (-4/3,1) circle (1.1pt);
			\fill [color=black]  (-2,1.5) circle (1.1pt);
			
			\fill [color=black] (-1.5,1.5) circle (1.1pt);
			\fill [color=black]  (-1.75,1.5) circle (1.1pt);
			
			\fill [color=black] (-2/5, 2) circle (1.1pt);
			
			\end{scriptsize}
			
			\end{tikzpicture}
			
			\caption{The triangle $\triangle oz_1z_2$}
			
			
		\end{centering}
		
	\end{figure}
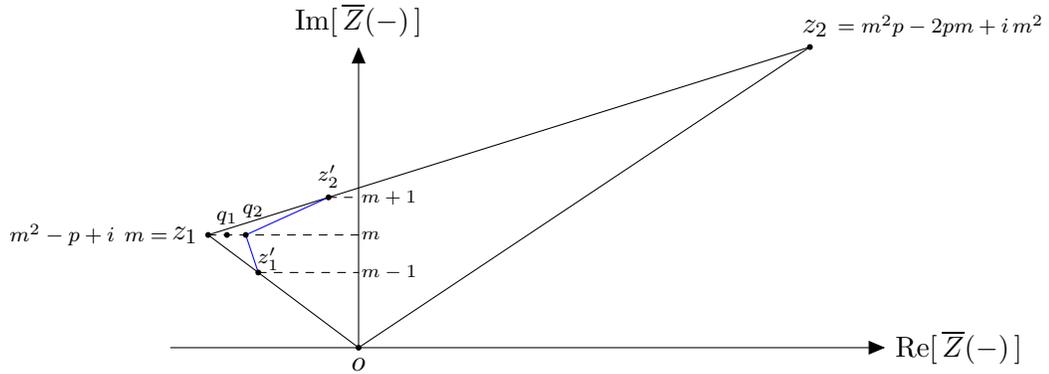
	
	If $P_{i_*F}$ is contained in $oz_1'q_2z_2'z_2$, then 
	\begin{align*}
	h^0(F) \leq \, & \frac{\chi(F)}{2} + \frac{\lVert \overline{oz_1'} \rVert+\lVert \overline{z_1'q_2} \rVert+\lVert \overline{q_2z_2'} \rVert+\lVert \overline{z_2'z_2} \rVert}{2}\\
	=\, &   \frac{2pm-m^2p}{2}+ \frac{1}{2}\sqrt{\frac{(m-1)^2(m^2-p)^2}{m^2} + (4p+4)(m-1)^2   } \\
	& \, + \frac{1}{2}\sqrt{  \left(\frac{m^2-p}{m} +2 \right)^2+ 4p+4}  +
	\frac{1}{2} \sqrt{  \left(\frac{p(m-1)}{m} - \frac{m}{m-1} -2\right)^2+ 4p+4} \\
	& \, + \frac{m^2-m-1}{2(m^2-m)}\sqrt{\big(p(m-1)^2-m^2 \big)^2 + (4p+4)(m^2-m)^2   }   \\
	\eqqcolon \, & h
	\end{align*}

	A direct computation gives $\lfloor h \rfloor = p+m^2 -1$ if $(p,m) = (19, 3), \ (29,4), \, (47, 5)$, or $(59, 7)$. So if $P_{i_*F}$ is contained in $\triangle oz_1z_2$, it must pass through $q_1$. If $(p,m)=(19,2), (47,5)$ or $(59,7)$, the points $q_1$ and $z_2$ have odd real part. The length of the first part of $P_{i_*F}$ between $o$ and $q_2$ is less than $\lVert \overline{oz_1'}\rVert + \lVert \overline{z_1'q_1}\rVert$ and the length of the second part between $q_1$ and $z_2$ is less than $\lVert \overline{q_1z_2'}\rVert + \lVert \overline{z_2'z_2}\rVert$. Applying Remark \ref{Rem.polygon} for these two parts gives
	\begin{equation*}
	h^0(F) \leq  \frac{\chi(F)-1}{2} + \left \lfloor \frac{\lVert \overline{oz_1'}\rVert + \lVert \overline{z_1'q_1}\rVert +1}{2} \right\rfloor + \left \lfloor \frac{\lVert q_1z_2'\rVert + \lVert z_2'z_2\rVert}{2} \right\rfloor \eqqcolon h'
	\end{equation*}
	In all cases, $h' = p+m^2 -1$, that is why $P_{i_*F}$ coincides with $\triangle oz_1z_2$. If $(p,m) = (29,4)$, then the real part of $q_1$ and $z_2$ are even, so 
	\begin{equation*}
	h^0(F) \leq  \frac{\chi(F)}{2} + \left \lfloor \frac{\lVert oz_1'\rVert + \lVert \overline{z_1'q_1}\rVert}{2} \right\rfloor + \left \lfloor \frac{\lVert \overline{q_1z_2'}\rVert + \lVert \overline{z_2'z_2}\rVert}{2} \right\rfloor = -116+22+138 =44
	\end{equation*}
	which is not possible. Thus $P_{i_*F}$ coincides with $\triangle oz_1z_2$.  
	
\end{proof}

\begin{Lem}\label{lem.23}
	The polygon $P_{i_*F}$ coincides with the triangle $\triangle oz_1z_2$ when $(p,m) = (23, 5)$.  
\end{Lem}

\begin{proof}
	In this case $z_1 = -2 + i 5$, $z_2 = 345 + i 25$, $z_2' = \frac{383}{20} + i 6$ and $48 \leq h^0(F)$. We show in few steps that if $P_{i_*F}$ is strictly inside $\triangle oz_1z_2$, then $h^0(F)$ is less than $48$. 
	
	\begin{figure}[h]
		\begin{centering}
			\definecolor{zzttqq}{rgb}{0.27,0.27,0.27}
			\definecolor{qqqqff}{rgb}{0.33,0.33,0.33}
			\definecolor{uququq}{rgb}{0.25,0.25,0.25}
			\definecolor{xdxdff}{rgb}{0.66,0.66,0.66}
			
			\begin{tikzpicture}[line cap=round,line join=round,>=triangle 45,x=1.0cm,y=1.0cm]

			\draw[->,color=black] (0,0) -- (0,4);
			\draw[->,color=black] (-2,0) -- (10.8,0);
			\draw[color=black] (1,2.5) -- (0, 0);
			\draw[color=black] (1, 2.5) -- (10.5,3.06);
			
			\draw[color=blue] (0, 0) -- (.5,1);
			\draw[color=blue] (.5, 1) -- (1.5,2);
			\draw[color=blue] (1.5, 2) -- (3.5,2.5);
			\draw[color=blue] (9.5, 3) -- (3.5,2.5);
			
			\draw[color=black, dashed] (9.5, 3) -- (9.5, 0);
			\draw[color=black, dashed] (9.5, 3) -- (0, 3);
			\draw[color=black, dashed] (0, 2.5) -- (3.5, 2.5);
			\draw[color=black, dashed] (0, 2.) -- (1.5, 2.);
			\draw[color=black, dashed] (0, 1.5) -- (1, 1.5);
			\draw[color=black, dashed] (0, 1) -- (.5, 1);
			\draw[color=black, dashed] (.5, 0) -- (.5, 2.5);
			\draw[color=black, dashed] (1, 0) -- (1, 2.5);
			\draw[color=black, dashed] (1.5, 0) -- (1.5, 2.5);
			\draw[color=black, dashed] (2, 0) -- (2, 2.5);
			\draw[color=black, dashed] (2.5, 0) -- (2.5, 2.5);
			\draw[color=black, dashed] (3, 0) -- (3, 2.5);
			\draw[color=black, dashed] (3.5, 0) -- (3.5, 2.5);

			\draw[color=black] (0, 0) -- (10.5, 1);

			\draw[color=black, dashed] (0, 0) -- (2.5,2.5);				
			%
			\draw (10.8,0) node [right] {Re$[\,\overline{Z}(-)\,]$};
			\draw (0,4) node [above] {Im$[\,\overline{Z}(-)\,]$};
			\draw (1.05,2.62) node [left] {$z_1$};
			\draw (9.5, 3) node [above] {$z_2'$};
			\draw (9.5, 0) node [below] {$\frac{383}{20}$};		
			
			\draw (.74,1.2)  node [below] {\fontsize{7}{2}{$q_0$}};
			\draw (1.26,1.7)  node [below] {\fontsize{7}{2}{$q_1$}};
			\draw (1.7, 2.1)  node [below] {\fontsize{7}{2}{$q_8$}};
			\draw (1.26,1.95)  node [above] {\fontsize{7}{2}{$q_2$}};
			\draw (1.6,2.5)  node [above] {\fontsize{7}{2}{$q_3$}};
			\draw (2.1,2.5)  node [above] {\fontsize{7}{2}{$q_4$}};
			\draw (2.6, 2.5)  node [above] {\fontsize{7}{2}{$q_5$}};
			\draw (3.1,2.5)  node [above] {\fontsize{7}{2}{$q_6$}};
			\draw (3.76, 2.55)  node [below] {\fontsize{7}{2}{$q_7$}};
			
			\draw (0,0) node [below] {$o$};
			\draw (.5,0) node [below] {\fontsize{7}{2}{$1$}};
			\draw (1,0) node [below] {\fontsize{7}{2}{$2$}};
			\draw (1.5,0) node [below] {\fontsize{7}{2}{$3$}};
			\draw (2,0) node [below] {\fontsize{7}{2}{$4$}};
			\draw (2.5,0) node [below] {\fontsize{7}{2}{$5$}};
			\draw (3,0) node [below] {\fontsize{7}{2}{$6$}};
			\draw (3.5,0) node [below] {\fontsize{7}{2}{$7$}};	
			
			\draw (0, 1) node [left] {\fontsize{7}{2}{$2$}};
			\draw (0, 1.5) node [left] {\fontsize{7}{2}{$3$}};
			\draw (0, 2) node [left] {\fontsize{7}{2}{$4$}};
			\draw (0, 2.5) node [left] {\fontsize{7}{2}{$5$}};
			\draw (0, 3) node [left] {\fontsize{7}{2}{$6$}};	
			
			\begin{scriptsize}
			
			\fill [color=black] (0,0) circle (1.1pt);
			\fill [color=black] (.5,1) circle (1.4pt);
			\fill [color=black] (1,1.5) circle (1.4pt);
			\fill [color=black]  (1,2) circle (1.4pt);
			\fill [color=black] (1.5,2) circle (1.4pt);
			\fill [color=black] (1, 2.5) circle (1.4pt);
			\fill [color=black] (1.5, 2.5) circle (1.4pt);
			\fill [color=black] (2,2.5) circle (1.4pt);
			\fill [color=black] (2.5,2.5) circle (1.4pt);
			\fill [color=black] (3, 2.5) circle (1.4pt);
			\fill [color=black] (3.5, 2.5) circle (1.4pt);
			\fill [color=black] (9.5, 3) circle (1.4pt);
			\end{scriptsize}
			
			\end{tikzpicture}
			
			\caption{The left hand side of $P_{i_*F}$ for $(p,m) = (23, 5)$.}
			
			\label{fig. p=23}
			
		\end{centering}
		
	\end{figure}
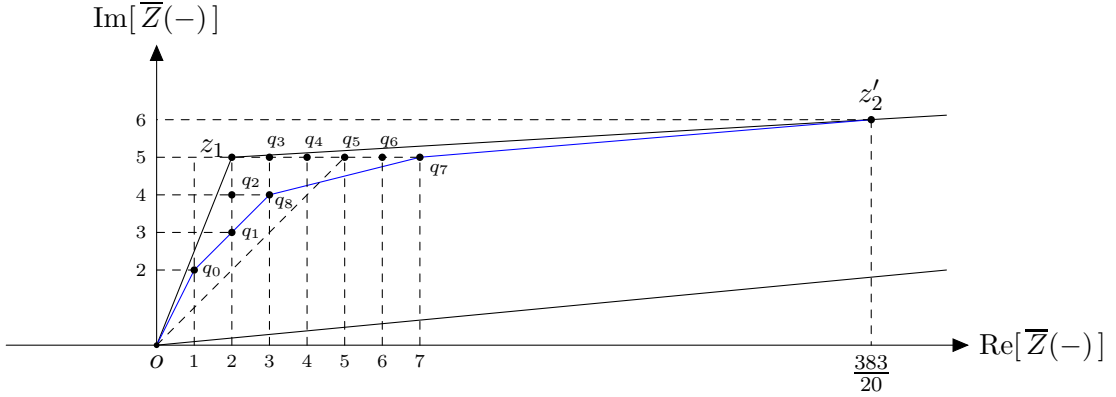
	
	\begin{enumerate}
		\item [\textbf{Step 1.}]
		If $P_{i_*F}$ is contained in the polygon $oq_{0}q_{1}q_{8}q_{7}z_2'z_2$ in Figure \ref{fig. p=23}, then 
		\begin{align*}
		h^0(F) \leq&\,  \frac{\chi(F)}{2} + \frac{\lVert \overline{oq_{0}} \rVert+\lVert \overline{q_{0}\,q_{8}} \rVert+\lVert \overline{q_{8}q_{7}} \rVert+\lVert \overline{q_{7}z_2'} \rVert+\lVert \overline{z_2'z_2} \rVert}{2}\\
		=&\,  \frac{-345}{2} + \frac{\sqrt{96\times 4+1 }}{2} + \frac{\sqrt{96\times 4+4 }}{2} +\frac{\sqrt{96+ 16 }}{2} +\\
		&\frac{1}{2}\sqrt{\left(\frac{383}{20} -7\right)^2+96} +\frac{1}{2}\sqrt{\left(345-\frac{383}{20}\right)^2 + 96 \times 19^2}\\
		< &\, 48
		\end{align*}
		\item [\textbf{Step 2.}] If $q_{6} = 6 + i 5$ is a vertex of $P_{i_*F}$, the length of $P_{i_*F}$ between $o$ and $q_6$ is less than $\lVert \overline{oq_{2}}\rVert + \lVert \overline{q_2q_6}\rVert$ and its length between $q_6$ and $z_2$ is less than $\lVert \overline{q_6z_2'}\rVert + \lVert \overline{z_2'z_2}\rVert$. Applying Remark \ref{Rem.polygon} for these two pieces gives  
		\begin{equation*}
		h^0(F) \leq \frac{\chi(F)-1}{2} + \left\lfloor \frac{\lVert \overline{oq_{2}}\rVert + \lVert \overline{q_2q_6}\rVert}{2}\right\rfloor +  \left\lfloor \frac{\lVert \overline{q_6z_2'}\rVert + \lVert \overline{z_2'z_2}\rVert   +1}{2}\right\rfloor  = -173+ 24 + 196= 47\ .
		\end{equation*}
		\item [\textbf{Step 3.}] If $q_5 = 5+i 5$ is a vertex of $P_{i_*F}$, the length of $P_{i_*F}$ between $q_5$ and $z_2$ is less than  
		$\lVert \overline{q_{5}z_2'}\rVert + \lVert \overline{z_2'z_2} \rVert$ and 
		\begin{equation}\label{ell.1}
		\left\lfloor \frac{\lVert \overline{q_{5}z_2'}\rVert + \lVert \overline{z_2'z_2} \rVert}{2} \right\rfloor = 196\ . 
		\end{equation} 
		For the first part of the polygon between $o$ and $q_5$ one of the following holds:
		\begin{enumerate*}
			\item $\overline{Z}(\tilde{E}_1) = k+ki$ for an integer $0 < k \leq 5$ where $\tilde{E}_1$ is the first non-trivial subobject in the HN filtration of $i_*F$. If $k=1$ or $k=4$, then Remark \ref{Rem.polygon} and \eqref{ell.1} give 
			\begin{equation*}
			h^0(F) \leq \frac{\chi(F) -1}{2} + \left\lfloor \frac{\sqrt{96 +1} +1}{2}\right\rfloor
			+ \left\lfloor \frac{\sqrt{16\times 96 +16}}{2}\right\rfloor + 196 =47 \ .
			\end{equation*}
			If $k=2$ or $k=3$, then 
			\begin{equation*}
			h^0(F) \leq \frac{\chi(F) -1}{2} + \left\lfloor \frac{\sqrt{4 \times 96 +4} }{2}\right\rfloor
			+ \left\lfloor \frac{\sqrt{9\times 96 +9} +1}{2}\right\rfloor + 196 = 47\ .
			\end{equation*}
			If $k=5$, then $v(\tilde{E}_1) = (r, 5H, r-5)$ for some $r \in \mathbb{Z}$. 
			We first show that 
			\begin{equation}\label{bound.11}
			h^0(\tilde{E}_1) \leq  \frac{\chi(\tilde{E}_1) -1}{2} + 24 \ .
			\end{equation}
			Consider the evaluation map $ \text{ev} \colon \mathcal{O}_X^{\oplus h^0(\tilde{E}_1)} \rightarrow \tilde{E}_1$. The same argument as in the proof of Lemma \ref{bound for h} implies cok(ev) is $\sigma$-semistable with respect to the stability condition $\sigma$ on the Brill-Noether wall. Let $\{W_i\}_{i=1}^{m}$ be the $\sigma$-stable factors of cok(ev). We know $0 \leq c(W_i) \leq c(\tilde{E}_1) = 5$. By reordering of the factors, we may assume $W_i \cong \mathcal{O}_X$ for $0 < i \leq i_0$, and $c(W_i) \neq 0$ for $i_0 < i \leq m$. Therefore there are at most 5 factors which are not isomorphic to the structure sheaf. Inequality \eqref{better bound in lemm} in the proof of Lemma \ref{bound for h} gives 
			\begin{equation}\label{bound.111}
			h^0(\tilde{E}_1) +i_0 \leq \left\lfloor \frac{\chi(\tilde{E}_1)}{2} + \frac{\sqrt{25 +96 \times 25}}{2} \right\rfloor= \frac{\chi(\tilde{E}_1) -1}{2} +25 
			\end{equation}    
			which implies the claim \eqref{bound.11} if $i_0 >0$. 
			So we may assume $i_0 =0$ and all stable factors satisfy $ 0< c(W_i) \leq 5$. By \eqref{bound.111}, we know $h^0(\tilde{E}_1) \leq \frac{\chi(\tilde{E}_1)-1}{2} + 25 = r+22$. Suppose for a contradiction $h^0(\tilde{E}_1) = r+22$. Then $v(\text{cok(ev)}) = (-22 , 5H, -27)$ is primitive, so stable factors $\{W_i\}_{i=1}^{m}$ cannot be all isomorphic. If $W_i \ncong  W_j$, then $\langle v(W_i)_i , v(W_j) \rangle \geq 0$ and if they are equal $W_i \cong W_j$, then $\langle v(W_i) , v(W_j) \rangle \geq -2$. Therefore the same argument as in \cite[Lemma 3.1]{feyz:mukai-program} implies that $\left(\sum_{i=1}^{m} v(W_i)\right)^2$ is minimum when four from five factors are equal, thus 
			\begin{equation*}
			-2(16 +1) \leq \left(\sum_{i=1}^{m} v(W_i)\right)^2  = \big(v(\text{cok(ev)})\big)^2 = -38\ ,
			\end{equation*}
			a contradiction. This completes the proof of \eqref{bound.11}. Combining it with \eqref{ell.1}
			gives $h^0(F) \leq \frac{\chi(F)-1}{2} + 24 + 196 = 47$ .
			
			\item $\overline{oq_0}$ is a side of $P_{i_*F}$, then the length of $P_{i_*F}$ between $q_2$ and $q_5$ is less than $\lVert \overline{q_0q_{2}}\rVert + \lVert \overline{q_{2}q_{5}} \rVert$, so Remark \ref{Rem.polygon} and \eqref{ell.1} give 
			\begin{align*}
			h^0(F) &\leq \frac{\chi(F)-1}{2} + \left\lfloor \frac{\lVert \overline{oq_0} \rVert +1 }{2}\right\rfloor +  \left\lfloor \frac{\lVert \overline{q_0q_{2}}\rVert + \lVert \overline{q_{2}q_{5}} \rVert  }{2}\right\rfloor + 196  \\
			&= -173+ 10+ 14+ 196= 47 \ . 
			\end{align*}
			\item $\overline{oq_{1}}$ is a side of $P_{i_*F}$, then 
			\begin{align*}
			h^0(F) &\leq \frac{\chi(F)-1}{2} + \left\lfloor \frac{\lVert \overline{oq_{1}} \rVert }{2}\right\rfloor +  \left\lfloor \frac{\lVert \overline{q_{1}q_{8}}\rVert + \lVert \overline{q_{8}q_{5}} \rVert+1  }{2}\right\rfloor + 196 \\
			&= -173+ 14+10+ 196= 47 \ .
			\end{align*}
			\item $\overline{oq_{8}}$ is a side of $P_{i_*F}$, then $v(\tilde{E}_1) = (r, 4H, r-3)$ for some $r \in \mathbb{Z}$. Lemma \ref{bound for h} implies that $h^0(\tilde{E}_1) \leq \frac{\chi(\tilde{E}_1)-1}{2} +\lfloor \frac{\sqrt{9 + 4 \times 23 \times 16+ 4} +1}{2} \rfloor = \frac{\chi(\tilde{E}_1)-1}{2} +19$. Thus  
			\begin{align*}
			h^0(F) \leq \frac{\chi(F)-1}{2} + 19 +  \left\lfloor \frac{ \lVert \overline{q_{8}q_{5}} \rVert }{2}\right\rfloor + 196 =-173+ 19+5+ 196= 47\ . 
			\end{align*}
			\item $\overline{oq_{2}}$ is a side of $P_{i_*F}$, then 
			\begin{align*}
			h^0(F) \leq \frac{\chi(F)-1}{2} + \left\lfloor \frac{\lVert \overline{oq_{2}} \rVert }{2}\right\rfloor +  \left\lfloor \frac{ \lVert \overline{q_{2}q_{5}} \rVert+1  }{2}\right\rfloor + 196 = -173+19+5+ 196= 47 \ .
			\end{align*}
		\end{enumerate*}
		
		\item [\textbf{Step 4.}] Suppose $q_{4} = 4 + i5$ is a vertex of $P_{i_*F}$. Let $q = a+ ib$ be the adjacent vertex which lies between $\overline{z_1z_2}$ and $\overline{q_{4}z_2}$. Let $q'_k = x'_k + i k $, $q''_k= x''_k + ik$ and $q'''_k = x'''_k+ik$ be the intersection points of the line $y=k$ with $\overline{z_1z_2}$, $\overline{q_{3}z_2}$ and $\overline{q_4z_2}$, respectively. Thus $x'_{b} \leq a \leq x'''_{b}$. The values of $x'_k$, $x''_k$ and $x'''_k$ where $5< k < 25$ are given in Table \ref{xk}.  
		
		\begin{table}[ht]
			\centering
			\begin{tabular}{|c|c|c|c|c|c|c|c|c|c|c|c|c|}
				\hline
				$k$& $6$&$7$&$8$&9&10&11 & $12$ & $13$ & $14$ & $15$& $16$ &17  	\\
				\hline
				$x'_k$&19.1 &36.3&53.4&70.6&87.7&104.9&122.05&139.2&156.3&173.5&190.6&207.8 \\
				\hline
				$x''_k$&20.1&37.2&54.3&71.4&88.5&105.6&122.7&139.8&156.9&174&191.1&208.2\\
				\hline
				$x'''_k $&21.05 &38.1&55.1&72.2&89.2&106.3&123.3&140.4&157.4&174.5&191.5&208.6\\
				\hline 	
			\end{tabular}
			
			\centering
			\begin{tabular}{|c|c|c|c|c|c|c|c|}
				\hline
				$k$& $18$&$19$&$20$&21&22&23 & $24$ 	\\
				\hline
				$x'_k$& 224.9&242.1&259.2&276.4&293.5&310.7&327.8 \\
				\hline
				$x''_k$&225.3&242.4&259.5&276.6&293.7&310.8&327.9\\
				\hline
				$x'''_k$& 225.6 &242.7&259.7&276.8&293.8&310.9&327.9\\
				\hline 	
			\end{tabular}

			\caption{Values of $x'_k$, $x''_k$ and $x'''_k$} 
			\label{xk}
		\end{table}
		
		The length of $P_{i_*F}$ between $o$ and $q_4$ is less than $\lVert \overline{oq_{2}} \rVert + \lVert \overline{q_{2}q_{4}} \rVert$ and 
		\begin{equation}\label{ell'}
		\left\lfloor \frac{\lVert \overline{oq_{2}} \rVert + \lVert \overline{q_{2}q_{4}} \rVert}{2}\right\rfloor= 24 \ .
		\end{equation}
		First assume $q = z_2$. Since there is no integral point on $(\overline{q_4z_2})$, the Mukai vector of the last factor in the HN filtration of $i_*F$ is $v(i_*F/\tilde{E}_{n-1}) = (r, 20 H, r-341)$. Lemma \ref{bound for h} gives 
		$$
		h^0(i_*F/\tilde{E}_{n-1}) \leq \left \lfloor  \frac{\chi(i_*F/\tilde{E}_{n-1})}{2} + \frac{\sqrt{341^2 + 4 \times 23 \times 20^2 + 4}}{2}  \right\rfloor =\left \lfloor  \frac{\chi(i_*F/\tilde{E}_{n-1})-1}{2} \right\rfloor + 196\ .
		$$  
		By combining it with \eqref{ell'}, one gets $h^0(F) \leq \frac{\chi(F)-1}{2} + 24+196 = 47$. Hence we may assume $q$ lies below $z_2$, i.e. $b < 25$. If the real part of $q$ is odd, then \eqref{ell'} gives   
		\begin{align}\label{in.1.bound for h}
		h^0(F) & \leq \frac{\chi(F)-1}{2} + 24 + \left\lfloor \frac{\lVert \overline{q_{4}q} \rVert +1}{2}\right\rfloor + \left\lfloor \frac{\lVert \overline{qq'_{b+1}} \rVert + \lVert \overline{q'_{b+1}z_2}  \rVert}{2}\right\rfloor\ .
		\end{align}
		All possible values of $q$ with odd real part are given in Table \ref{fig.odd.1}. 	
		
		\begin{table}[ht]
			\centering
			\begin{tabular}{|c|c|c|c|c|c|c|}
				\hline
				$q$ & $21+ 6i$ & $37+7i$ & $55+8i$ & $71+ 9i$ & $89+10i$& $105+11i$  	\\
				\hline
				$\left\lfloor \frac{\lVert \overline{q_{4}q} \rVert +1}{2}\right\rfloor$&10 &19&29&39&49&58 \\
				\hline
				$\left\lfloor \frac{\lVert \overline{qq'_{b+1}} \rVert + \lVert \overline{q'_{b+1}z_2}  \rVert}{2}\right\rfloor $&186 &177&167&157&147&138\\
				\hline 	
			\end{tabular}
			
			\begin{tabular}{|c|c|c|c|c|}
				\hline
				$q$ &  $123+ 12i$&$157+14i$&$191+ 16i$&$225+18i$   	\\
				\hline
				$\left\lfloor \frac{\lVert \overline{q_{4}q} \rVert +1}{2}\right\rfloor$&69&88&108&128\\
				\hline
				$\left\lfloor \frac{\lVert \overline{qq'_{b+1}} \rVert + \lVert \overline{q'_{b+1}z_2}  \rVert}{2}\right\rfloor $&127&108&88&$69^*$\\
				\hline 	
			\end{tabular}
			\caption{$q$ has odd real part} 
			\label{fig.odd.1}
		\end{table}
		
		In all cases, except $q = 225+ 18i$, one gets $h^0(F) \leq 47$ by \eqref{in.1.bound for h}. Table \ref{xk} shows that there is no integral point between line segments $\overline{z_1z_2}$, $\overline{q_4z_2}$ and above the line $y= 18$. That is why if $q= 225+ 18i$, then $\overline{qz_2}$ is a side of polygon corresponding to the last factor $i_*F/\tilde{E}_{n-1}$ in the HN filtration of $i_*F$. We have $v(i_*F/\tilde{E}_n) = (r, 7H , r-120)$ for some $r \in \mathbb{Z}$, so Lemma \ref{bound for h} gives 
		\begin{equation*}
		h^0(i_*F/\tilde{E}_{n-1}) \leq \left\lfloor  \frac{\chi(i_*F/\tilde{E}_{n-1})}{2} + \frac{1}{2} \sqrt{120^2 + 4 \times 23 \times 49 +4} \right\rfloor =  \frac{\chi(i_*F/\tilde{E}_{n-1})}{2} + 68
		\end{equation*}
		Therefore $h^0(F) \leq \frac{\chi(F)-1}{2} + 24 + 128 + 68 = 47$.

		Similarly, if the real part of $q$ is even, then
		\begin{equation}\label{in.2.bound for h}
		h^0(F) \leq -173 + 24 + \left\lfloor \frac{\lVert \overline{q_{4}q} \rVert}{2}\right\rfloor + \left\lfloor \frac{\lVert \overline{qq'_{b+1}} \rVert + \lVert \overline{q'_{b+1}z_2}  \rVert+1}{2}\right\rfloor
		\end{equation}
		
		All possible values of $q$ with even real part are given in Table \ref{fig.even.1}.
		
		\begin{table}[ht]
			\centering
			\begin{tabular}{|c|c|c|c|c|c|}
				\hline
				$q$ & $20+ 6i$ & $38+7i$ & $54+8i$ & $72+ 9i$ & $88+10i$\\
				\hline
				$\left\lfloor \frac{\lVert \overline{q_{4}q} \rVert}{2}\right\rfloor  $&9 &19&28&39&48 \\
				\hline
				$\left\lfloor \frac{\lVert \overline{qq_{b+1}} \rVert + \lVert \overline{q_{b+1}z_2}  \rVert+1}{2}\right\rfloor$&187 &177&168&157&148\\
				\hline 	
			\end{tabular}
			
			
			\begin{tabular}{|c|c|c|c|c|}
				\hline
				$q$ & $106+11i$ & $140+ 13i$&$174+15i$&$208+ 17i$   	\\
				\hline
				$\left\lfloor \frac{\lVert \overline{q_{ 4}q} \rVert}{2}\right\rfloor  $&58&78&98&117 \\
				\hline
				$\left\lfloor \frac{\lVert \overline{qq_{b+1}} \rVert + \lVert \overline{q_{b+1}z_2}  \rVert+1}{2}\right\rfloor$&138&118&99*&79\\
				\hline 	
			\end{tabular}
			
			\caption{$q$ has even real part} 
			\label{fig.even.1}
		\end{table}
		
		In all cases, except $q = 174+i \ 15$, one gets $h^0(F) \leq 47$ by \eqref{in.2.bound for h}. Suppose $q= 174+15i$. If $\overline{qz_2}$ is a side of polygon, then it corresponds to the last factor $i_*F/\tilde{E}_{n-1}$ in the HN filtration. We have $v(i_*F/\tilde{E}_{n-1}) = (r, 10H, r-171)$ for some $r \in \mathbb{Z}$. Thus Lemma \ref{bound for h} gives 
		$h^0(i_*F/\tilde{E}_{n-1}) \leq \left\lfloor \frac{\chi(i_*F/\tilde{E}_{n-1})}{2} + \frac{\sqrt{ 171^2 + 4 \times 23 \times 100 +4 }}{2} \right\rfloor =  \frac{\chi(i_*F/\tilde{E}_{n-1})-1}{2} +98$ which implies $h^0(F) \leq 47$. If there is another extremal point $q'$ between $q$ and $z_2$, then Table \ref{xk} shows that one of the following cases happens: 
		\begin{enumerate}
			\item  $q' = 191+ 16i$, then 
			\begin{align*}
			h^0(F) &\leq -149 + 	\left\lfloor \frac{\lVert \overline{q_{4}q} \rVert}{2}\right\rfloor + 	\left\lfloor \frac{\lVert \overline{qq'} \rVert +1}{2}\right\rfloor + 	\left\lfloor \frac{\lVert q'q'_{17} \rVert+ \lVert \overline{q'_{17}z_2} \rVert}{2}\right\rfloor \\
			&= -149+98+10+88=47 \ .
			\end{align*}
			\item  $q' =  208+ i17$, then 
			\begin{align*}
			h^0(F) &\leq -149 + 	\left\lfloor \frac{\lVert \overline{q_{4}q} \rVert}{2}\right\rfloor + 	\left\lfloor \frac{\lVert \overline{qq'} \rVert}{2}\right\rfloor + 	\left\lfloor \frac{\lVert \overline{q'q'_{18}} \rVert+ \lVert \overline{q'_{18}z_2} \rVert+1}{2}\right\rfloor \\
			& = -149+98+19+79=47\ .
			\end{align*}
			\item $q' =  225+ i18$, then 
			\begin{align*}
			h^0(F) &\leq -149 + 	\left\lfloor \frac{\lVert \overline{q_{4}q} \rVert}{2}\right\rfloor + 	\left\lfloor \frac{\lVert \overline{qq'} \rVert +1}{2}\right\rfloor + 	\left\lfloor \frac{\lVert \overline{q'q'_{19}} \rVert+ \lVert \overline{q'_{19}z_2} \rVert}{2}\right\rfloor  \\
			&= -149+98+29+69=47\ .
			\end{align*}
		\end{enumerate}
		
		\item [\textbf{Step 5.}] Suppose $q_3 = 3+i5$ is a vertex of $P_{i_*F}$. There is a subobject $\tilde{E}_k$ in the HN filtration of $i_*F$ such that $\overline{Z}(\tilde{E}_k) = q_3 = -3+i5$. We first show that 
		\begin{equation}\label{h'}
		h^0(\tilde{E}_k) \leq \frac{\chi(\tilde{E}_k) -1}{2} + 24\ .
		\end{equation} 
		\begin{enumerate}
			\item If $\overline{oq_3}$ is a side of $P_{i_*F}$, then $v(\tilde{E}_1) = (r, 5H, r-3)$ for some $r \in \mathbb{Z}$. Thus Lemma \ref{bound for h} implies the claim. 
			\item If $\overline{oq_0}$ is a side $P_{i_*F}$, but $\overline{q_0q_2}$ is not its side, then 
			\begin{equation*}
			h^0(\tilde{E}_k) \leq \frac{\chi(\tilde{E}_k) -1}{2} +\left\lfloor \frac{\lVert \overline{oq_0} \rVert +1}{2} \right\rfloor
			+\left\lfloor \frac{\lVert \overline{q_0q_3} \rVert }{2} \right\rfloor = \frac{\chi(\tilde{E}_k) -1}{2} + 10 +14 \ . 
			\end{equation*}
			\item If $\overline{oq_2}$ is a side of $P_{i_*F}$, then 
			\begin{equation*}
			h^0(\tilde{E}_k) \leq \frac{\chi(\tilde{E}_k) -1}{2} +\left\lfloor \frac{\lVert \overline{oq_2} \rVert }{2} \right\rfloor
			+\left\lfloor \frac{\lVert \overline{q_2q_3} \rVert +1 }{2} \right\rfloor = \frac{\chi(\tilde{E}_k) -1}{2} + 19 +5 \ . 
			\end{equation*}
		\end{enumerate} 
		Now let $q = a+bi$ be the adjacent vertex to $q_3$ which lies between $\overline{z_1z_2}$ and $\overline{q_3z_2}$.
		If $q = z_2$, the last factor in the HN filtration of $i_*F$ has Mukai vector $(r, 20H, r-342)$ for some $r \in \mathbb{Z}$. Thus Lemma \ref{bound for h} gives $$h^0(i_*F/\tilde{E}_{n-1}) \leq \frac{\chi(i_*F/\tilde{E}_{n-1})}{2} + \left\lfloor \frac{\sqrt{342^2 + 4 \times 23 \times 20^2 + 4 \time 2^2}}{2}\right\rfloor = \frac{\chi(i_*F/\tilde{E}_{n-1})}{2} +196\ .$$ Combining it with \eqref{h'} gives $h^0(F) \leq 47$. So we may assume $q$ is below $z_2$, i.e. $b < 25$. If $q$ has odd real part, then 
		\begin{equation}\label{h''}
		h^0(i_*F/\tilde{E}_k) \leq \frac{\chi(i_*F/\tilde{E}_k)}{2} + \left\lfloor \frac{\lVert \overline{q_3q} \rVert }{2} \right\rfloor + \left\lfloor \frac{\lVert \overline{qq'_{b+1}}\rVert +\lVert \overline{q'_{b+1}z_2} \rVert}{2} \right\rfloor \ .
		\end{equation} 
		All possible values of $q$ with odd real part are given in Table \ref{table.odd.1}.  
		\begin{table}[ht]
			\centering
			\begin{tabular}{|c|c|c|c|c|c|}
				\hline
				$q$ &  $37+7i$ & $71+ 9i$ & $105+11i$& $191+16i$&$225+18i$  	\\
				\hline
				$\left\lfloor \frac{\lVert \overline{q_{3}q} \rVert}{2}\right\rfloor$&19 &39&58&108&127 \\
				\hline
				$\left\lfloor \frac{\lVert \overline{qq'_{b+1}} \rVert + \lVert \overline{q'_{b+1}z_2}  \rVert}{2}\right\rfloor $&177 &157&138&88&69\\
				\hline 	
			\end{tabular}
			\caption{$q$ has odd real part} 
			\label{table.odd.1}
		\end{table}
		
		In all cases, $\left\lfloor \frac{\lVert \overline{oq_2} \rVert }{2} \right\rfloor
		+\left\lfloor \frac{\lVert \overline{q_2q_3} \rVert +1 }{2} \right\rfloor = 196$. Thus combining \eqref{h'} and \eqref{h''} gives $h^0(F) \leq 47$. 
		
		If $q$ has even real part, then 
		\begin{equation*}
		h^0(i_*F/\tilde{E}_k) \leq \frac{\chi(i_*F/\tilde{E}_k)}{2}-1  + \left\lfloor \frac{\lVert \overline{q_3q} \rVert +1}{2} \right\rfloor + \left\lfloor \frac{\lVert \overline{qq'_{b+1}}\rVert +\lVert \overline{q'_{b+1}z_2} \rVert+1}{2} \right\rfloor \ .
		\end{equation*} 
		All possible values of $q$ with even real part are given in Table \ref{table.even.1}.
		
		\begin{table}[ht]
			\centering
			\begin{tabular}{|c|c|c|c|c|c|}
				\hline
				$q$ &  $20+6i$ & $54+ 8i$ & $88+10i$& $174+15i$&$208+17i$  	\\
				\hline
				$\left\lfloor \frac{\lVert q_{3}q \rVert +1}{2}\right\rfloor$&10 &29&49&99*&118 \\
				\hline
				$\left\lfloor \frac{\lVert qq'_{b+1} \rVert + \lVert q'_{b+1}z_2  \rVert+1}{2}\right\rfloor $&187 &168&148&99*&79\\
				\hline 	
			\end{tabular}
			\caption{$q$ has even real part} 
			\label{table.even.1}
		\end{table}
		
		In all cases except $q = 174+15i$, one gets $h^0(F) \leq 47$. Note that $q = 174+15i$ lies on $\overline{q_3z_2}$, so convexity of $p_{i_*F}$ gives $\overline{qz_2}$ is its side. The Mukai vector of stable factors corresponding to the sides $\overline{q_3q}$ and $\overline{qz_2}$ are of the form $(r_i , 10 H, r_i -171)$ for some $r_i \in \mathbb{Z}$. Thus Lemma \ref{bound for h} for these two factors gives
		\begin{equation*}
		h^0(i_*F/\tilde{E}_k) \leq \frac{\chi(i_*F/\tilde{E}_k)}{2} -1 + 2\left\lfloor  \frac{\sqrt{171^2 + 4 \times 23 \times 10^2 + 4}}{2} + \frac{1}{2} \right\rfloor \ .
		\end{equation*}	
		Combining it with \eqref{h'} implies $h^0(F) \leq 46$.
		\item [\textbf{Step 6.}] If $q_2 = 2+i4$ is a vertex of $P_{i_*F}$ and none of the above cases happens, then $\overline{oq_2}$ is a side of $P_{i_*F}$ and its length between $q_2$ and $z_2$ is less than $\lVert \overline{q_2q_7} \rVert + \lVert \overline{q_7z_2'} \rVert + \lVert \overline{z_2'z_2}\rVert$. Therefore 
		\begin{align*}
		h^0(F) &\leq \frac{\chi(F)-1}{2} + \left\lfloor \frac{\lVert \overline{oq_2} \rVert}{2} \right\rfloor + \left\lfloor \frac{\lVert \overline{q_2q_7}\rVert +\lVert \overline{q_7z_2'} \rVert+\lVert \overline{z_2'z_2} \rVert+1}{2} \right\rfloor  \\
		&= -173 + 19+ 201 =47 \ .
		\end{align*} 
		\item  [\textbf{Step 7.}] If $q_8 = 3 + i4$ is a vertex of $P_{i_*F}$ and none of the above cases happens, then 
		\begin{align*}
		h^0(F) &\leq \frac{\chi(F)-1}{2} + \left\lfloor \frac{\lVert \overline{oq_0} \rVert + \lVert \overline{q_0q_8} \rVert +1}{2} \right\rfloor + \left\lfloor \frac{\lVert \overline{q_8q_7}\rVert +\lVert \overline{q_7z_2'} \rVert+\lVert \overline{z_2'z_2} \rVert}{2} \right\rfloor \\
		& = -173 + 20+ 200 =47 \ .
		\end{align*}  
		\item  [\textbf{Step 8.}] If $q_1 = 2 + i3$ is a vertex of $P_{i_*F}$ and none of the above happens, the length of $P_{i_*F}$ between $q_1$ and $z_2$ is less than $\lVert \overline{q_1\tilde{q}} \rVert + \lVert \overline{\tilde{q}q_7} \rVert + \lVert \overline{q_7z_2'} \rVert + \lVert \overline{z_2'z_2}\rVert$, where $\tilde{q} = 4+4i$. Thus  
		\begin{align*}
		h^0(F) \leq &  \frac{\chi(F)-1}{2} + \left\lfloor \frac{\lVert \overline{oq_0} \rVert + \lVert \overline{q_0q_1} \rVert }{2} \right\rfloor + \left\lfloor \frac{\lVert \overline{q_1\tilde{q}} \rVert + \lVert \overline{\tilde{q}q_7'} \rVert + \lVert \overline{q_7z_2'} \rVert + \lVert \overline{z_2'z_2}\rVert +1}{2} \right\rfloor  \\
		= & -173 + 14+206 =47\ .
		\end{align*}   
		
	\end{enumerate}
\end{proof}

\begin{proof}[Proof of Theorem \ref{theorem 1.1}]    	
	%
	%
	%
	%
	%
	A similar argument as in \cite[Propsoition 4.4]{feyz:mukai-program} gives the injectivity of $\psi$ as a result of the uniqueness of the Harder-Narasimhan filtration. Lemmas \ref{lem. large}, \ref{lem. 13}, \ref{lem.17}, \ref{lem.19} and \ref{lem.23} show that for any semistable vector bundle $F \in \mathcal{BN}$, the polygon $P_{i_*F}$ coincides with the triangle $\triangle oz_1z_2$. Since 
	\begin{equation*}
	\gcd(m, p-m^2) =1, 
	\end{equation*}
	there is no integral point on the open line segment $(\overline{oz_1})$. Therefore the first factor $\tilde{E_1}$ in the HN filtration of $i_*F$ has Mukai vector $v(\tilde{E_1}) = (m^2-k ,mH, p-k )$ for some $k \in \mathbb{Z}$. Since $\tilde{E_1}$ is $\sigma_{(0,w)}$-semistable for $w < w^*$, \cite[Lemma 3.1]{feyz:mukai-program} gives $v(\tilde{E}_1)^2 \geq -2m^2$ which implies $0 \leq k $. By Proposition \ref{5.3}, the line passing through $pr(v(\tilde{E}_1)) = \big(m/(p-k)\,,\,(m^2-k)/(p-k)\big)$ and $pr(v(i_*F))$ lies above or on the line segment $\overline{p_up_v}$, hence its slope is bigger than the slope of $\overline{p_up_v}$, i.e. 
	\begin{equation*}
	\frac{\frac{m^2-k}{p-k}}{\frac{m}{p-k} + \frac{m}{p(m-2)}} >  \frac{\frac{m^2}{p}}{\frac{m}{p} + \frac{m}{p(m-2)}} \ . 
	\end{equation*} 
	This gives $k=0$ and $v(\tilde{E}_1) = v$. Thus phase of the subobject $\tilde{E}_1$ is bigger than $i_*F$ for the stability conditions above $\overline{p_up_v}$, hence the wall for $i_*F$ lies below or on the line segment $\overline{p_up_v}$. Then Proposition \ref{5.3} implies that $F$ is the restriction of $\tilde{E}_1 \in M_{X,H}(v)$ to the curve $C$, so $\psi$ is surjective. Therefore Propositions \ref{5.2} gives any vector bundle $F$ in $\mathcal{BN}= M_C(m^2,2pm,p+m^2)$ is slope-stable and $h^0(F) = p+m^2$. Thus the same argument as in \cite[Theorem 1.2]{feyz:mukai-program} shows $\psi$ is an isomorphism.   	
\end{proof}



\bibliography{mybib}
\bibliographystyle{halpha}

\end{document}